\documentclass[11pt,centertags,reqno]{amsart}

\usepackage{amssymb,amsmath,amsfonts,amssymb}
\usepackage{hyperref}
\usepackage{enumerate}

-\marginparwidth \oddsidemargin -9mm \evensidemargin \oddsidemargin

\usepackage{latexsym}
\advance\hoffset by 5mm

\def\@abssec#1{\vspace{.05in}\footnotesize \parindent .2in
{\bf #1. }\ignorespaces}

\newtheorem{theorem}{Theorem}[section]

\newtheorem{lemma}[theorem]{Lemma}

\newtheorem{remark}[theorem]{Remark}

\newcommand{\EQ}[1]{\eqref{#1}}

\newcounter{hypo}

\newcommand{\R}{\ensuremath{\mathbb{R}}}

\newcommand{\N}{\ensuremath{\mathbb{N}}}
\newcommand{\LL}{\ensuremath{\mathcal{L}}}

\newcommand{\beq}[1]{\begin{equation}\label{#1}}
\newcommand{\eeq}{\end{equation}}
\newcommand{\beqs}{\begin{equation*}}
\newcommand{\eeqs}{\end{equation*}}
\newcommand{\set}[1]{\left\{#1\right\}}

\newcommand{\lp}{\left(}
\newcommand{\rp}{\right)}

\newcommand{\tend}{\bibliographystyle{plain}\bibliography{ccrituniq}\end{document}}
\newcommand{\dd}{\mathrm{d}}

\allowdisplaybreaks \voffset=-0.2in
\numberwithin{equation}{section}

\begin{document}
\bigskip

\title[Differentiability issue of drift-diffusion equation]{On the Differentiability issue of the drift-diffusion equation with nonlocal L\'evy-type diffusion}
\author{Liutang Xue}
\address{School of Mathematical Sciences, Beijing Normal University, and Laboratory of Mathematics and Complex Systems, Ministry of Education, Beijing 100875, P.R. China}
\email{xuelt@bnu.edu.cn}
\author{Zhuan Ye}
\address{School of Mathematical Sciences, Beijing Normal University, and Laboratory of Mathematics and Complex Systems, Ministry of Education, Beijing 100875, P.R. China}
\email{yezhuan815@126.com}
\date{}
\maketitle

\begin{abstract}
  We investigate the differentiability issue of the drift-diffusion equation with nonlocal L\'evy-type diffusion at either supercritical or critical type cases.
Under the suitable conditions on the drift velocity and the forcing term in terms of the spatial H\"older regularity,
we prove that the vanishing viscosity solution is differentiable with some H\"older continuous derivatives for any positive time.
\end{abstract}

\section{Introduction}\label{intro}
We consider the following drift-diffusion equation with nonlocal diffusion
\begin{equation}\label{DD}
\begin{cases}
  \partial_{t}\theta+(u \cdot \nabla) \theta+\mathcal{L}\theta=f, &\quad \textrm{in}\;\,\R^d\times \R^+,\\
  \theta(x,0)=\theta_{0}(x), &\quad  \textrm{on}\;\,\R^d,
\end{cases}
\end{equation}
where $\theta$ is a scalar function, $u$ is a velocity vector field of $\R^d$ and $f$ is a scalar function as the forcing term.
The nonlocal diffusion operator $\mathcal{L}$ is given by
\begin{equation}\label{L}
  \mathcal{L}\theta(x)={\rm{p.v.}}\int_{\mathbb{R}^d}{\big(\theta(x)-\theta(x+y)\big)K(y)\,\dd y},
\end{equation}
where the symmetric kernel function $K(y)=K(-y)$ defined on $\mathbb{R}^d\setminus\{0\}$ satisfies that
\begin{equation}\label{Kcd1}
  \int_{\R^d} \min \set{1, |y|^2} |K(y)|\,\dd y \leq c_1,
\end{equation}
and there exist two constants $\alpha\in(0,\,1]$ and $\sigma\in [0,\,\alpha)$ such that
\begin{equation}\label{Kcd2}
  \frac{c_2^{-1}}{|y|^{d+\alpha-\sigma}}\leq K(y)\leq \frac{c_2}{|y|^{d+\alpha}},\quad \forall 0<|y|\leq 1,
\end{equation}
with $c_1>0$ and $c_2\geq 1$ two absolute constants. Besides, in the sequel we also consider the kernel $K$ satisfying the nonnegative condition
\begin{equation}\label{Kcd3}
  K(y)\geq 0,\quad \forall y\in \R^d\setminus\{0\}.
\end{equation}

The nonlocal diffusion operator $\mathcal{L}$ defined by \eqref{L} with the symmetric kernel $K$ satisfying \eqref{Kcd1}-\eqref{Kcd2}
corresponds to the L\'evy-type operator, which is the infinitesimal generator of the stable-type L\'evy process  (cf. \cite{CSZ,Sato}).
By taking the Fourier transform on $\mathcal{L}$, we get
$$\widehat{\mathcal{L}\theta}(\xi)=A(\xi)\widehat{\theta}(\xi),$$
where the symbol $A(\xi)$ is given by the following L$\rm \acute{e}$vy-Khinchin formula
\begin{equation}\label{LKf}
  A(\xi)=\,\textrm{p.v.}\int_{\R^d}\lp 1-\cos(x\cdot\xi)\rp K(x)\dd x.
\end{equation}

The considered operator $\LL$ includes a large class of multiplier operators $\LL=A(D)=A(|D|)$ such as
\begin{equation}\label{Llog}
  \LL= \frac{|D|^\alpha}{\left(\log(\lambda + |D|)\right)^\mu},\;\lp\alpha\in (0,1],\mu\geq 0,\lambda>0\rp
\end{equation}
with $|D|:=(-\Delta)^{\frac{1}{2}}$, and one can refer to \cite[Lemmas 5.1-5.2]{DKSV} for more details on the assumptions of $A(\xi)$
so that the corresponding kernel $K$ satisfies \eqref{Kcd1}-\eqref{Kcd2};
we also note that the condition \eqref{Kcd3} can be satisfied under some additional assumption of $A(\xi)$, e.g., for all $\lambda\geq \lambda_0$ with $\lambda_0>0$ some number,
the operator \eqref{Llog} satisfies \eqref{Kcd1}-\eqref{Kcd3} (cf. \cite{Hmidi,MX15,DL14}).
If $\mu=0$ in \eqref{Llog}, the operator $\LL$ reduces to an important special case $ |D|^\alpha:= (-\Delta)^{\frac{\alpha}{2}}$ ($\alpha\in ]0,1]$),
usually called as the fractional Laplacian operator, which has the following expression formula
\begin{equation}\label{LamAlp}
  |D|^\alpha \theta(x)= c_{d,\alpha}\,\textrm{p.v.}\int_{\R^d} \frac{\theta(x)-\theta(x+y)}{|y|^{d+\alpha}} \dd y,
\end{equation}
with $c_{d,\alpha}>0$ some absolute constant. The operator $\LL =|D|^\alpha$ ($\alpha\in(0,2)$) is the infinitesimal generator of the symmetric stable L\'evy process (cf. \cite{Sato}),
and recently has been intensely considered in many theoretical problems.
For the drift-diffusion equation \EQ{DD} with $\LL=|D|^\alpha$, 
we conventionally call the cases $\alpha<1$, $\alpha=1$ and $\alpha>1$ as supercritical, critical and subcritical cases, respectively.
Thus the operator $\LL$ defined by \eqref{L} under the kernel conditions \eqref{Kcd1}-\eqref{Kcd2} can be viewed as the critical and supercritical type cases and is the main concern in this paper.

For the drift-diffusion equation \eqref{DD} with the fractional Laplacian operator $\LL= |D|^\alpha$, Silvestre in \cite{Silv} considered the supercritical and critical cases ($\alpha\in (0,1]$),
and proved the interior $C^{1,\gamma}$ regularity of the solution provided that $u$ and $f$ belong to $L^\infty_t C^{1-\alpha+\gamma}_x$ ($\gamma\in(0,\alpha)$),
more precisely, the author showed the following regularity estimate
\begin{equation}\label{SilEs}
  \|\theta\|_{L^\infty([-\frac{1}{2},0]; C^{1,\gamma}(B_{1/2}))}\leq C \lp\|u\|_{L^\infty([-1,0]\times\R^d)} + \|f\|_{L^\infty([-1,0]; C^{1-\alpha+\gamma}(B_1))} \rp,
\end{equation}
where $C>0$ depends only on $d,\alpha$ and $\|u\|_{L^\infty([-1,0]; C^{1-\alpha+\gamma})}$.
The proof is by a locally approximate procedure where an extension derived in \cite{CafS} plays a key role.
We note that if the velocity field is divergence-free,
a similar $C^{1,\gamma}$ regularity improvement of weak solution had previously obtained by Constantin and Wu in \cite{CW2008} by using the Bony's paradifferential calculus.
For the drift-diffusion equation \eqref{DD} with general diffusion operator, Chen et al in \cite{CSZ} considered the case that $\theta_0\equiv 0$ and
$\LL$ is defined by \eqref{L}-\eqref{Kcd2} (in fact for slightly more general operator $\LL$),
and by applying the probabilistic method, the authors proved the $C^{1,\gamma}$ regularity of a continuous solution under the condition that
$u$ and $f$ are $\dot C^\delta_x$ ($\delta\in(1-\alpha+\sigma,1)$) H\"older continuous for each time.

If we slightly lower the regularity index in the assumption of $u$ and $f$, the solution of the equation \eqref{DD}-\eqref{L} may in general not have such a differentiable regularity.
For the drift-diffusion equation \eqref{DD} with $\LL=|D|^\alpha$, Silvestre in \cite{Silv2} proved that if $u\in L^\infty_t\dot C^{1-\alpha}_x$ for $\alpha\in (0,1)$
and $u\in L^\infty_{t,x}$ for $\alpha=1$,
and if $f\in L^\infty_{t,x}$, then the bounded solution becomes H\"older continuous for any positive time.
For the drift-diffusion equation \eqref{DD} with more general $\LL$, and under the divergence-free condition of $u$,
we refer to \cite{ChamM} for a similar improvement to H\"older continuous solution (see also \cite{MaeM} for a related result).
Note that the condition $u\in L^\infty_t \dot C^{1-\alpha}$ is invariant under the the scaling transformation
$u(x,t)\mapsto \lambda^{\alpha-1} u(\lambda^\alpha t, \lambda x)$ for all $\lambda>0$.
If we further weaken the regularity condition on $u$ in the supercritical case, the solution of \eqref{DD}-\eqref{L} may not even be continuous,
indeed, as proved by Silvestre et al in \cite{SVZ}, there is a divergence-free drift $u\in L^\infty_t C^\delta_x$ (for every $\delta<1-\alpha$)
so that the solution of the equation \eqref{DD} with $\LL=|D|^\alpha$ and $f=0$ forms a discontinuity starting from smooth initial data.

In this paper, we are concerned with the differentiability of the vanishing viscosity solution (i.e. the solution derived from \eqref{appDD} by passing $\epsilon\rightarrow 0$)
for the system \eqref{DD}-\eqref{L}. We impose no regularity assumption on the initial data,
and we generalize the result of Silvestre \cite{Silv} for more general L\'evy-type operator.
Our first result is about the drift-diffusion equation \eqref{DD} under the kernel conditions \eqref{Kcd1}-\eqref{Kcd3},
and the velocity field needs not to be divergence-free.

\begin{theorem}\label{thm1}
  Let the symmetric kernel $K(y)=K(-y)$ of the diffusion operator $\mathcal{L}$ satisfy the conditions (\ref{Kcd1})-(\ref{Kcd3}).
Suppose that $\theta_{0}\in C_0(\mathbb{R}^d)$, and for $T>0$ any given, the drift $u$ and the external force $f$ satisfy
\begin{equation}\label{ufcd1}
  u\in L^\infty([0,T]; \dot C^\delta(\R^d)),\quad\textrm{and}\quad f\in L^{\infty}([0,T]; C^{\delta}(\mathbb{R}^d)),\quad \textrm{for every  } \delta\in(1-\alpha+\sigma,\,1),
\end{equation}
then the drift-diffusion equation \EQ{DD}-\EQ{L} admits a vanishing viscosity solution $\theta\in L^\infty([0,T]; C_0(\R^d))$
which satisfies $\theta\in L^{\infty}((0,\,T],\,C^{1,\gamma}(\mathbb{R}^d))$ for some constant $\gamma \in (0,\delta+\alpha-\sigma-1)$. Moreover, for any $\tilde{t}\in (0,T)$, we have
\begin{equation}\label{target1}
  \|\theta\|_{L^{\infty}([\tilde{t},T]; C^{1,\gamma}(\mathbb{R}^d))}\leq C \lp\|\theta_{0}\|_{L^{\infty}} +  \|f\|_{L^\infty_T C^\delta}\rp ,
\end{equation}
where $C$ is a positive constant depending only on $\tilde{t}$, $T$, $\alpha$, $\sigma$, $d$, $\delta$ and $\|u\|_{L^\infty_T \dot C^{\delta}}$.
\end{theorem}

Our second result states that if the velocity field is divergence-free, then the differentiability result can be achieved for the drift-diffusion equation under conditions \eqref{Kcd1}-\eqref{Kcd2},
without imposing the nonnegative condition \eqref{Kcd3}.

\begin{theorem}\label{thm2}
  Let the symmetric kernel $K(y)=K(-y)$ of the diffusion operator $\mathcal{L}$ satisfy (\ref{Kcd1})-(\ref{Kcd2}),
and the velocity field $u$ be divergence-free.
Assume that for $T>0$ any given, the drift $u$, the force $f$ and the initial data $\theta_0$ satisfy
\begin{equation}\label{ucd2}
  u \in L^{\infty}([0,T],\dot C^{\delta}(\mathbb{R}^d)), \quad\textrm{for every  }\delta\in(1-\alpha+\sigma,1),
\end{equation}
and
\begin{equation}\label{fcd2}
  f\in L^{\infty}([0,T]; B_{p,\infty}^\delta\cap B_{\infty,\infty}^\delta(\mathbb{R}^d)),\quad \theta_0\in L^p(\mathbb{R}^d),
  \quad\textrm{for every  }p\in [2,\infty).
\end{equation}
Then the drift-diffusion equation \eqref{DD}-\eqref{L} admits a vanishing viscosity solution $\theta\in L^\infty([0,T]; L^p(\R^d))$ which satisfies
$\theta\in L^{\infty}((0,T],C^{1,\gamma}(\mathbb{R}^d))$ with some constant $\gamma \in (0,\delta+\alpha-\sigma-1)$. Moreover, for every $t'\in (0,T)$, we have
\begin{equation}\label{target2}
  \|\theta\|_{L^{\infty}([t',T];C^{1,\gamma}(\mathbb{R}^d))}\leq C\lp \|\theta_0\|_{L^p} + \|f\|_{L^\infty_T (B_{p,\infty}^\delta\cap B_{\infty,\infty}^\delta)}\rp,
\end{equation}
with the constant $C$ depending only on $t'$,\,$T$,\,$\alpha$, $\sigma$, $d$, $\delta$ and $\|u\|_{L^\infty_T \dot C^{\delta}}$.
\end{theorem}

The method in showing Theorems \ref{thm1} and \ref{thm2} is consistent with the method of paradifferential calculus used in \cite{CW2008},
but is mostly in a different style;
and by applying the technique of time function weighted estimate (where Lemma \ref{lem-key} is of great use),
we find that the process used here is not sensitive to the divergence-free condition of $u$
so that we can get rid of such a condition in Theorem \ref{thm1}. We use the $L^\infty$-framework in proving Theorem \ref{thm1}
and the $L^p$ ($p\in [2,\infty)$)-framework in Theorem \ref{thm2},
and the key diffusion effect of the L\'evy-type diffusion operator (for high frequency part) is derived in Lemma \ref{lem:MPloc}
and Lemma \ref{lem:Lp1} respectively. The iterative argument also plays an important role in the proof of both theorems.

We also note that the approach of \cite{Silv} is not adopted here,
and it seems rather hard (if not possible) to extend the method of \cite{Silv} for the drift-diffusion equation with more general diffusion operator.

\begin{remark}[On higher regularity]
By examining the proof of both theorems, we see that the index $\gamma$ indeed can be any number belonging to $(0,\delta+\alpha-\sigma-1)$,
which is achieved by pursuing the iteration process for more times. In fact, for Theorem \ref{thm1}, the worst scenario is that there is no $s\in (1-\delta,\alpha-\sigma)$ so that $\tilde{s}+s=1+\gamma$
after obtaining the estimate of $\|\theta\|_{L^\infty B^{\tilde{s}}_{\infty,\infty}}$ with $1<\tilde{s}<\delta+\alpha-\sigma$,
but we can instead start with $L^\infty B^{\tilde{s}'}_{\infty,\infty}$ for some $\tilde{s}'<\tilde{s}$ so that we can get the improvement $\|\theta\|_{L^\infty B^{\tilde{s}'+s}_{\infty,\infty}}$
with $\tilde{s}'+s=1+\gamma$;
while for Theorem \ref{thm2}, for any $\gamma\in(0,\delta+\alpha-\sigma-1)$, there exists some $\tilde{p}<\infty$ so that $\gamma +d/\tilde{p}<\delta+\alpha-\sigma-1$,
thus the target is to obtain the bound of $\|\theta\|_{L^\infty B^{1+\gamma +d/\tilde{p}}_{\tilde{p},\infty}}$,
which can be deduced from a more direct iterative process due to the increment $s\in(0,\alpha-\sigma)$.
Moreover, if \eqref{ufcd1} and \eqref{ucd2}-\eqref{fcd2} hold for any $\delta> 1-\alpha+\sigma$ by removing the restriction $\delta<1$,
we infer that the vanishing viscosity solution studied in Theorems \ref{thm1} and \ref{thm2} satisfies
\begin{equation*}
  \theta\in
  \begin{cases}
  L^\infty((0,T]; C^{[\delta+\alpha-\sigma]-1,\gamma}),\;\;\forall \gamma\in (0,1), & \quad \textrm{if   }\delta+\alpha-\sigma\in\N, \\
    L^\infty((0,T]; C^{[\delta+\alpha-\sigma],\gamma}),\;\; \forall\gamma\in (0,\delta+\alpha-\sigma-[\delta+\alpha-\sigma]),&\quad \textrm{if   }\delta+\alpha-\sigma\notin\N^+.
  \end{cases}
\end{equation*}

As a consequence of the above result, if $f=0$ and $u=\mathcal{P}\theta$ in the equation \eqref{DD} with $\mathcal{P}$ composed of zero-order pseudo-differential operators
(e.g. the SQG equation in \cite{CW2008}: $d=2$, $u=(-\mathcal{R}_2\theta, \mathcal{R}_1\theta)$ with $\mathcal{R}_j$, $j=1,2$ the usual Riesz transform),
we can deduce that under the assumptions of Theorems \ref{thm1} and \ref{thm2},
the corresponding solution belongs to $C^\infty((0,T]\times \R^d)$. Indeed, after obtaining the bound of $\|\theta\|_{L^\infty C^{1,\gamma}}$
(and $\|\theta\|_{L^\infty B^{1+\gamma +d/\tilde{p}}_{\tilde{p},\infty}}$ with some $\tilde{p}<\infty$ in Theorem \ref{thm2}) for any $\gamma\in (0,\delta+\alpha-\sigma-1)$,
from the Calder\'on-Zygmund theorem, we get $\nabla u\in L^\infty \dot C^\gamma$, which further leads to
\begin{equation*}
  \theta\in
  \begin{cases}
  L^\infty C^{[1+\gamma+\alpha-\sigma]-1,\gamma'},\;\;\forall \gamma'\in (0,1), & \quad \textrm{if   }1+\gamma+\alpha-\sigma\in\N, \\
    L^\infty C^{[1+\gamma+\alpha-\sigma],\gamma'},\;\; \forall\gamma'\in (0,\gamma+\alpha-\sigma-[\gamma+\alpha-\sigma]),&\quad \textrm{if   }1+\gamma+\alpha-\sigma\notin\N^+,
  \end{cases}
\end{equation*}
(in Theorem \ref{thm2} we in fact obtain a stronger estimate on $\theta$ in terms of $L^p$-based Besov spaces);
noting that the regularity index can be arbitrarily close to $\delta+2(\alpha-\sigma)$ by suitably choosing $\gamma$ and $\gamma'$,
thus by the bootstrapping method, we can iteratively improve the regularity and finally conclude the $C^\infty$-smoothness of the solution.
\end{remark}

\begin{remark}[The case $\delta=1-\alpha+\sigma$]
By examining the proof of both theorems, we have the following results.

$\bullet$    In Theorem \ref{thm1}, if the drift $u$ obeys the smallness condition
$$\|u\|_{_{L^{\infty}([0,T];C^{1-\alpha+\sigma}(\mathbb{R}^d))}}\leq \epsilon$$ for some absolute constant $\epsilon>0$ small enough independent of $T$ and $f\in L^{\infty}([0,T]; C^{1-\alpha+\sigma}(\mathbb{R}^d))$, then for any $\nu\in (0,\,1)$ and $\hat{t}\in(0,\,T)$ it holds that
$$ \|\theta\|_{L^{\infty}([\hat{t},T];C^{\nu}(\mathbb{R}^d))}\leq C(\|\theta_0\|_{L^\infty} + \|f\|_{L^\infty_T B_{\infty,\infty}^{1-\alpha+\sigma}}),$$
where $C$ is a positive constant depending only on $\hat{t}$, $T$, $\alpha$, $\sigma$, $d$ and $\|u\|_{L^\infty_T \dot C^{1-\alpha+\sigma}}$.

$\bullet$    This result is also true for Theorem \ref{thm2}. More precisely,
if the drift $u$ obeys the smallness condition
$$\|u\|_{_{L^{\infty}([0,T];C^{1-\alpha+\sigma}(\mathbb{R}^d))}}\leq  \widetilde{\epsilon}$$ for some absolute constant $\widetilde{\epsilon}>0$ small enough independent of $T$ and $f\in L^{\infty}([0,T]; B_{p,\infty}^{1-\alpha+\sigma}\cap B_{\infty,\infty}^{1-\alpha+\sigma}(\mathbb{R}^d))$, then for any $\widetilde{\nu}\in (0,\,1)$ and $\widetilde{t}\in (0,\,T)$ it holds that
$$\|\theta\|_{L^{\infty}([\widetilde{t},T];C^{\widetilde{\nu}}(\mathbb{R}^d))}\leq C  (\|\theta_0\|_{L^p} + \|f\|_{L^\infty_T (B_{p,\infty}^{1-\alpha+\sigma}\cap B_{\infty,\infty}^{1-\alpha+\sigma})}),$$
with the constant $C$ depending only on $\widetilde{t}$,\,$T$,\,$\alpha$, $\sigma$, $d$ and $\|u\|_{L^\infty_T \dot C^{1-\alpha+\sigma}}$.

\end{remark}

The outline of the paper is as follows. In Section \ref{sec:Prel}, we present some preliminary knowledge on Bony's paradifferential calculus and the Besov spaces.
Section \ref{sec:thm1} is dedicated to the proof of Theorem \ref{thm1}: we first show some useful auxiliary lemmas, then we prove the key a priori estimate \eqref{target1}
in the whole subsection \ref{subsec:thm1-2}, and then we sketch the existence part and conclude the theorem. We show Theorem \ref{thm2} in Section \ref{sec:thm2},
and the proof is also divided into three parts: the auxiliary lemmas, the a priori estimates and the existence issue,
which are treated in the subsections \ref{subsec:thm2-1} - \ref{subsec:thm2-3} respectively.

Throughout this paper, $C$ stands for a constant which may be different from line to line.
The notion $X\lesssim Y$ means that $X\leq CY$, and $X\approx Y$ implies that $X\lesssim Y$ and $Y\lesssim X$ simultaneously.
Denote $\mathcal{S}'(\mathbb{R}^d)$ the space of tempered distributions, $\mathcal{S}(\mathbb{R}^d)$ the Schwartz space of rapidly
decreasing smooth functions, $\mathcal{S}'(\mathbb{R}^d)/\mathcal{P}(\mathbb{R}^d)$ the quotient space of tempered distributions which modulo polynomials.
We use $\widehat{g}$ of $\mathcal{F}(g)$ to denote the Fourier transform of a tempered distribution, that is, $\widehat g(\xi)=\int_{\mathbb{R}^d}e^{i x\cdot \xi} g(x)\mathrm{d} x$.
For a number $a\in \mathbb{R}$, denote by $[a]$ the integer part of $a$.

\section{Preliminaries}\label{sec:Prel}
\setcounter{equation}{0}
In this preliminary section, we shall collect some basic facts on the Bony's paradifferential calculus and the Besov spaces.

First we recall the so-called Littlewood-Paley operators and their elementary properties.
Let $(\chi, \varphi)$ be a couple of smooth functions taking values on $[0, 1]$ such that
$\chi\in C_0^\infty(\mathbb{R}^d)$ is supported in the ball $\mathcal{B}:= \{\xi\in \mathbb{R}^d, |\xi|\leq \frac{4}{3}\}$,
$\varphi\in C_0^{\infty}(\mathbb{R}^d)$ is supported in the annulus $\mathcal{C}:= \{\xi\in \mathbb{R}^d, \frac{3}{4}\leq |\xi|\leq \frac{8}{3}\}$
and satisfies that (cf. \cite{BCD})
$$\chi(\xi)+\sum_{j\in \mathbb{N}}\varphi(2^{-j}\xi)=1, \; \forall \xi\in \mathbb{R}^d,\quad \textrm{and}
\quad\sum_{j\in \mathbb{{Z}}}\varphi(2^{-j}\xi)=1,  \; \forall  \xi \in \mathbb{R}^d\setminus \{0 \}.$$
For every $ u\in S'(\R^d)$,  we define the non-homogeneous Littlewood-Paley operators as follows,
$$\Delta_{-1}f=\chi(D)u; \quad \, \quad\Delta_{j}f=\varphi(2^{-j}D)f,\;\;S_j f=\sum_{-1\leq k\leq j-1} \Delta_{k}u,\;\;\;\forall j\in \mathbb{N}.$$
And the homogeneous Littlewood-Paley operators can be defined as follows
\begin{equation*}
  \dot{\Delta}_j f:= \varphi(2^{-j}D)f;\quad \dot S_j f:= \sum_{k\in\mathbb{Z},k\leq j-1}\dot \Delta_k f, \quad \forall j\in\mathbb{Z}.\quad
\end{equation*}
Also, we denote
$$\widetilde{\Delta}_j f:=\Delta_{j-1}f+\Delta_j f+\Delta_{j+1}f.$$
It is clear to see that, for any $f$ and $g$ belonging to $S'(\R^d)$, from the property of the frequency supports, we have
$$\Delta_j \Delta_l f\equiv 0, \quad |j-l|\geq 2 \quad \textrm{and}
\quad \Delta_k (S_{l-1} g\Delta_l g)\equiv 0 \quad |k-l|\geq 5.$$

Now we introduce the definition of Besov spaces.
Let $s\in \mathbb{R}, (p,r)\in[1,+\infty]^2$, then the inhomogeneous Besov space $B_{p,r}^s$ is defined as
\begin{equation*}
  B^s_{p,r}:=\Big\{f\in\mathcal{S}'(\mathbb{R}^d);\|f\|_{B^s_{p,r}}:=\|\{2^{js}\|\Delta
  _j f\|_{L^p}\}_{j\geq -1}\|_{\ell^r }<\infty  \Big\},
\end{equation*}
and the homogeneous space $\dot B^s_{p,r}$ is given by
\begin{equation*}
  \dot{B}^s_{p,r}:=\Big\{f\in\mathcal{S}'(\mathbb{R}^d)/\mathcal{P}(\mathbb{R}^d);
  \|f\|_{\dot{B}^s_{p,r}}:=\|\{2^{js}\|\dot{\Delta}_j f\|_{L^p}\}_{j\in\mathbb{Z}}\|_{\ell^r(\mathbb{Z})}<\infty  \Big\}.
\end{equation*}
For any non-integer $s>0$, the H\"{o}lder space $C^s=C^{[s],s-[s]}$ is equivalent to $B^s_{\infty,\infty}$ with $\|f\|_{C^s}\approx \|f\|_{B^s_{\infty,\infty}}$.

Bernstein's inequality plays an important role in the analysis involving Besov spaces.
\begin{lemma}[cf. \cite{BCD}]
Let $f\in L^a$, $1\leq a\leq b\leq \infty$. Then for every $(k,j)\in\mathbb{N}^2$, there exists a constant $C>0$ independent of $j$ such that
\begin{equation*}
 \sup_{|\alpha|=k}\set{\partial^\alpha S_j f}_{L^b}\leq C 2^{j(k+\frac{d}{a}-\frac{d}{b})}\|S_j f\|_{L^a},
\end{equation*}
and
\begin{equation*}
  C^{-1}2^{j k}\|\Delta_j f\|_{L^a}\leq \sup_{|\alpha|=k}\|\partial^\alpha \Delta_j f\|_{L^a}\leq C 2^{j k}\|\Delta_j f\|_{L^a}.
\end{equation*}
\end{lemma}

\section{Proof of Theorem \ref{thm1}}\label{sec:thm1}
\setcounter{equation}{0}

\subsection{Auxiliary lemmas}

Before proceeding the main proof, we introduce several crucial auxiliary lemmas. First is the usual maximum principle for the drift-diffusion equation \eqref{DD}-\eqref{L}.
\begin{lemma}\label{lem01}
Let $u$ be a smooth vector field and $f$ be a smooth forcing term. Assume that $\theta$ is a smooth solution for the drift-diffusion \eqref{DD}-\eqref{L} with $\theta_0\in C_0(\R^d)$
and the assumptions of $K$ \eqref{Kcd1}-\eqref{Kcd3}. Then for $T>0$, we have
\begin{eqnarray}\label{MP1}
  \max_{0\leq t\leq T}\|\theta(t)\|_{L^{\infty}}\leq \|\theta_{0}\|_{ L^{\infty}}+\int_0^T \|f(t)\|_{ L^{\infty}}\dd t.
\end{eqnarray}
\end{lemma}
\begin{proof}[Proof of Lemma \ref{lem01}]
  Since we have the nonnegative condition \eqref{Kcd3}, the proof is quite similar to \cite[Theorem 4.1]{CorC}, and we thus omit the details.
\end{proof}

The second is the maximum principle with diffusion effect for the following frequency localized drift-diffusion equation
\begin{equation}\label{flDD}
  \partial_t \Delta_j\theta+ u\cdot\nabla \Delta_j\theta  + \mathcal{L}\Delta_j\theta=g,
\end{equation}
where $j\in\N$, the operator $\mathcal{L}$ defined by \eqref{L} with the symmetric kernel $K$ satisfying \eqref{Kcd1}-\eqref{Kcd3}.
\begin{lemma}\label{lem:MPloc}
  Assume that $u$ and $f$ are smooth functions, and $\theta$ is a smooth solution for the equation \eqref{flDD} with $\Delta_j \theta\in C_0(\R^d)$ for all $t>0$ and $j\in\N$.
Then there exist two absolute positive constants $c$ and $C$ depending only on $\alpha,\sigma,d$ such that
\begin{equation}\label{keyest1}
  \partial_t\|\Delta_j\theta\|_{L^\infty}+c\,2^{j(\alpha-\sigma)}\|\Delta_j\theta\|_{L^\infty} \leq C\|\Delta_j \theta\|_{L^\infty}+\|g\|_{L^\infty}.
\end{equation}
\end{lemma}

\begin{proof}[{Proof of Lemma \ref{lem:MPloc}}]
  Denote by $\theta_j:=\Delta_j\theta$, and from $\theta_j(t)\in C_0(\R^d)$ for $j\in\N$, there exists a point
$x_{t,j}\in \R^d$ so that $|\theta_{j}(t,x_{t,j})|=\|\theta_{j}\|_{L^{\infty}}>0$. Without loss of generality, we assume $\theta_j(t,x_{t,j})=\|\theta_j\|_{L^\infty}>0$
(otherwise, we consider the equation of $-\theta_j$ and replace $\theta_j$ by $-\theta_j$ in the following deduction).
Now by using \eqref{L}, \eqref{Kcd3}, \eqref{LamAlp} and the estimate $\theta(t,x_{t,j})-\theta(t,x_{t,j}+y)\geq 0$, we get
\begin{eqnarray}\label{OK001}
  \mathcal{L}\theta_j(x_{t,j})&=&
  \textrm{p.v.}\int_{\mathbb{R}^d}{\big(\theta_{j}(x_{t,j})-\theta_{j}(x_{t,j}+y) \big)K(y)\,\dd y} \nonumber\\&=&
  \textrm{p.v.} \int_{|y|\leq1}{\big(\theta_{j}(x_{t,j})-\theta_{j}(x_{t,j}+y) \big)K(y)\,\dd y}+
  \int_{|y|>1}{\big(\theta_{j}(x_{t,j})-\theta_{j}(x_{t,j}+y) \big)K(y)\,\dd y}\nonumber\\&\geq&
  c_2^{-1}\,\textrm{p.v.} \int_{|y|\leq1}{\frac{ \theta_{j}(x_{t,j})-\theta_{j}(x_{t,j}+y) }{|y|^{d+\alpha-\sigma}}\,\dd y}+
  \int_{|y|>1}{\big(\theta_{j}(x_{t,j})-\theta_{j}(x_{t,j}+y) \big)K(y)\,\dd y} \nonumber\\&\geq&
  c_2^{-1} \,\textrm{p.v.}\int_{\mathbb{R}^d}{\frac{\theta_j(x_{t,j})-\theta_j(x_{t,j}+y) }{|y|^{d+\alpha-\sigma}}\,\dd y}-c_2^{-1}\int_{|y|>1}{\frac{ \theta_{j}(x_{t,j})-\theta_{j}(x_{t,j}+y) }{|y|^{d+\alpha-\sigma}}\,\dd y}\nonumber\\&\geq& c_2^{-1} c_{d,\alpha}^{-1} \, |D|^{\alpha-\sigma}\theta_{j}(x_{t,j})-2 c_2^{-1}\|\theta_{j}\|_{L^{\infty}}
  \int_{|y|>1}{\frac{1}{|y|^{d+\alpha-\sigma}}\,dy}\nonumber\\&\geq&c_2^{-1} c_{d,\alpha}^{-1} \,|D|^{\alpha-\sigma}\theta_{j}(x_{t,j})-C\|\theta_{j}\|_{L^{\infty}}.
\end{eqnarray}
According to \cite[Lemma 3.4]{WangZh}, we have
\begin{eqnarray}\label{OK002}
  |D|^{\alpha-\sigma}\theta_j(x_{t,j})\geq \tilde{c} 2^{j(\alpha-\sigma)} \|\theta_j\|_{L^{\infty}},
\end{eqnarray}
with some $\tilde{c}>0$. Inserting (\ref{OK002}) into (\ref{OK001}) yields
\begin{eqnarray}\label{OK003}
  &&\mathcal{L}\theta_{j}(x_{t,j})\geq c\,2^{j(\alpha-\sigma)} \|\theta_{j}\|_{L^{\infty}}-C\|\theta_{j}\|_{L^{\infty}}.
\end{eqnarray}
Hence, by arguing as \cite[Lemma 3.2]{WangZh} and using the fact $\nabla\theta_j(t,\,x_{t,j})=0$, we get
\begin{eqnarray}\label{OK004}
  \partial_{t}\|\theta_{j}\|_{L^{\infty}} = \partial_{t}\theta_{j}(t,\,x_{t,j}) &=&
  -u(t,\,x_{t,j})\cdot\nabla \theta_{j}(t,\,x_{t,j})-\mathcal{L}\theta_{j}(t,\,x_{t,j})+g(t,\,x_{t,j}) \nonumber\\
  &\leq &- c2^{j(\alpha-\sigma)} \|\theta_j\|_{L^{\infty}}+C\|\theta_{j}\|_{L^{\infty}}+\|g\|_{L^\infty},
\end{eqnarray}
which finishes the proof of \eqref{keyest1}.
\end{proof}

\begin{lemma}\label{lem-key}
Let $\lambda>0$ and $0<l<1$. Then for any $t>0$, there exists a constant $C_l$
depending only on $l$ such that
\begin{eqnarray}\label{key-eq1}
  \int_0^t e^{-(t-\tau)\lambda}\tau^{-l}\,\dd\tau \leq C_l\lambda^{-1} t^{-l}.
\end{eqnarray}
In particular, for any $t>t_0 \geq 0$, we have
\begin{eqnarray}\label{key-eq2}
  \int_{t_0}^t e^{-(t-\tau)2^{(\alpha-\sigma) j}}(\tau-t_0)^{-l}\,\dd\tau=\int_0^{t-t_0} e^{-(t-t_0-\tau)
  2^{(\alpha-\sigma) j}}\tau^{-l}\,\dd\tau \leq C_l 2^{-(\alpha-\sigma) j}(t-t_0)^{-l}.
\end{eqnarray}
\end{lemma}
\begin{proof}[{Proof of Lemma \ref{lem-key}}]
First, by changing of the variable $(t-\tau)\lambda=s$, one deduces
\begin{eqnarray}
  \int_0^t e^{-(t-\tau)\lambda}\tau^{-l}\,\dd\tau\nonumber  &=&
  \lambda^{-1} \int_0^{t\lambda} e^{-s}\Big(t-\frac{s}{\lambda}\Big)^{-l}\,\dd s\nonumber\\ &=&
  \lambda^{-1}\left(\int_0^{\frac{t\lambda}{2}} e^{-s} \Big(t-\frac{s}{\lambda}\Big)^{-l}\,\dd s
  + \int_{\frac{t\lambda}{2}}^{t\lambda} e^{-s} \Big(t-\frac{s}{\lambda}\Big)^{-l}\,\dd s\right)\nonumber\\
  &=&\lambda^{-1}(B_1+B_2).\nonumber
\end{eqnarray}
For the first term $B_1$, noting that $t-\frac{s}{\lambda}\geq \frac{1}{2}t$ for all $0\leq s\leq\frac{t\lambda}{2}$, we directly get
\begin{eqnarray}
  B_1 \leq 2^l t^{-l} \int_0^{\frac{t\lambda}{2}} e^{-s}\,\dd s\nonumber
  \leq 2^l t^{-l}\int_0^\infty e^{-s} \,\dd s\nonumber
  \leq  C 2^{l}t^{-l}.\nonumber
\end{eqnarray}
For the second term $B_2$, by changing of the variable $t-\frac{s}{\lambda}= s'$ and using the fact $t\lambda e^{-\frac{t\lambda}{2}}\leq C$, we deduce that
\begin{eqnarray}
  B_2 & \leq & e^{-\frac{t\lambda}{2}}\int_{\frac{t\lambda}{2}}^{t\lambda}
  \Big(t-\frac{s}{\lambda}\Big)^{-l}\,\dd s\nonumber \\
  & = & t^{1-l}\lambda e^{-\frac{t\lambda}{2}}\int_0^{\frac{1}{2}}  (s')^{-l}\,\dd s'\nonumber
  =\frac{2^{l-1}}{1-l} t^{-l}\Big(t\lambda e^{-\frac{t\lambda}{2}}\Big)\nonumber
  \leq\frac{ C 2^{l-1}}{1-l} t^{-l}.\nonumber
\end{eqnarray}
Combining the above two estimates, we obtain
$$\int_0^t e^{-(t-\tau)\lambda}\tau^{-l}\,\dd\tau\leq C\Big(2^l+\frac{2^{l-1}}{1-l} \Big) \lambda^{-1}t^{-l}=: C_l \lambda^{-1}t^{-l},$$
which concludes \eqref{key-eq1}.
\end{proof}

\subsection{A priori estimates}\label{subsec:thm1-2}

In this subsection, we assume $\theta$ is a smooth solution having suitable spatial decay for the drift-diffusion equations \eqref{DD}-\eqref{L}
with smooth $u$ and smooth $f$.
We intend to show the estimate \eqref{target1} and the proof is divided into three steps.


\textbf{Step 1: }the estimation of $\|\theta\|_{L^\infty([t_0,T];C^s(\R^d))}$ for any $s \in (1-\delta,\alpha-\sigma)$ and $t_0\in (0,T)$.

For every $j\in\mathbb{N}$, applying the inhomogeneous dyadic operator $\Delta_j$ to the considered equation (\ref{DD}), we obtain
\begin{eqnarray}\label{flEq1}
  \partial_{t}\Delta_j\theta+ u\cdot\nabla \Delta_j\theta +\mathcal{L}\Delta_j\theta=u\cdot\nabla \Delta_j\theta -\Delta_j(u\cdot\nabla \theta)+\Delta_j f.
\end{eqnarray}
Bony's paraproduct decomposition leads to
\begin{eqnarray}\label{Idec}
  u\cdot\nabla\Delta_j\theta-\Delta_j(u\cdot\nabla\theta)&=&-\sum_{|k-j|\leq4}  [ \Delta_j, S_{k-1}u\cdot\nabla ] \Delta_k\theta  \nonumber\\
  &&-\sum_{|k-j|\leq4}\Big( \Delta_j\big(\Delta_k u\cdot\nabla S_{k-1}\theta\big)-\Delta_k u\cdot\nabla\Delta_j S_{k-1}\theta \Big) \nonumber\\
  && -\sum_{k\geq j-2} \Big(\Delta_j\big(\Delta_k u\cdot\nabla \widetilde\Delta_k \theta\big)-\Delta_k u\cdot\nabla\Delta_j \widetilde\Delta_k\theta\Big) \nonumber\\
  &:=&I_1+I_2+I_3,
\end{eqnarray}
where in the first line we used the notation of commutator $[A,B]=AB-BA$ for two operators $A$ and $B$.
Taking advantage of Lemma \ref{lem:MPloc} in the frequency localised equation \eqref{flEq1}, we get
\begin{eqnarray}\label{thejLinf1}
  \frac{d}{dt} \|\Delta_j\theta\|_{L^\infty}+c 2^{j(\alpha-\sigma)}
  \|\Delta_j\theta\|_{L^\infty}\leq C_1\|\Delta_j \theta\|_{L^\infty}+\|I_1\|_{L^\infty}+\|I_2\|_{L^\infty}
  +\|I_3\|_{L^\infty}+\|\Delta_j f\|_{L^\infty}.
\end{eqnarray}
For $\|I_1\|_{L^\infty}$, noting that $I_1$ can be expressed as
\begin{equation}\label{I1exp}
  I_1 = -\sum_{|k-j|\leq4}\int_{\mathbb{R}^d}  \phi_j(x-y)\big(S_{k-1}u(y)-S_{k-1}u(x)\big)\cdot\nabla\Delta_k\theta(y)\,\dd y,
\end{equation}
where $\phi_j(x)=2^{jd}(\mathcal{F}^{-1}\varphi)(2^{j}x)$ and $\varphi\in C_0^{\infty}(\mathbb{R}^d)$ is the test function introduced in Section \ref{sec:Prel}, thus from the H\"older and Bernstein inequalities, one has
\begin{eqnarray}\label{I1Linf}
  \|I_1\|_{L^\infty} &\leq&  \sum_{|k-j|\leq4}\Big\| \int_{\mathbb{R}^d}\phi_j(x-y)\big(S_{k-1}u(y)-S_{k-1}u(x)\big)
  \cdot\nabla\Delta_k\theta(y)\,\dd y \Big\|_{L_x^\infty}\nonumber\\
  &\leq&  C\sum_{|k-j|\leq4} \Big\| \int_{\mathbb{R}^d} |\phi_j(x-y)|\, \|u\|_{\dot C^\delta}|x-y|^\delta \,|\nabla\Delta_k\theta(y)|\,\dd y \Big\|_{L_x^\infty} \nonumber\\
  & \leq& C 2^{-j\delta}\|u\|_{\dot C^\delta}\sum_{|k-j|\leq4}2^k\|\Delta_k\theta\|_{L^\infty}.
\end{eqnarray}
By virtue of H\"older's inequality and Bernstein's inequality again, we also see that
\begin{eqnarray}\label{I2Linf}
  \|I_2\|_{L^\infty} &\leq& \sum_{|k-j|\leq 4} \|\Delta_j(\Delta_k u\cdot\nabla S_{k-1}\theta)\|_{L^\infty}+
  \sum_{|k-j|\leq 4}\|\Delta_k u\cdot\nabla S_{k-1}\Delta_j\theta\|_{L^\infty} \nonumber\\
  &\leq& C \sum_{|k-j|\leq 4}\|\Delta_k u \|_{L^\infty} \|\nabla S_{k-1}\theta\|_{L^\infty} + C \sum_{|k-j|\leq 4} \|\Delta_k u\|_{L^\infty} \|\nabla \Delta_j \theta\|_{L^\infty}\nonumber\\
  &\leq& C 2^{-j\delta}\sum_{|k-j|\leq 4} 2^{k\delta} \|\Delta_k u\|_{L^\infty} \lp \sum_{l\leq j}2^l \|\Delta_l\theta\|_{L^\infty} \rp \nonumber\\
  &\leq& C 2^{-j\delta} \|u\|_{\dot C^\delta} \lp \sum_{k\leq j}2^k\|\Delta_k\theta\|_{L^\infty}\rp,
\end{eqnarray}
and
\begin{eqnarray}\label{I3Linf}
  \|I_3\|_{L^\infty} &\leq& \sum_{k\geq j-2}\Big\|\Delta_j(\Delta_k u\cdot\nabla \widetilde\Delta_k\theta)\Big\|_{L^\infty}
  +  \sum_{k\geq j-2}\Big\|  \Delta_k u\cdot\nabla \widetilde\Delta_k\Delta_j\theta\Big\|_{L^\infty}\nonumber\\
  &\leq& C \sum_{k\geq j-2}\|\Delta_k u\|_{L^\infty}2^k\|\widetilde\Delta_k\theta\|_{L^\infty} \nonumber\\
  &\leq& C  \sum_{k\geq j-2}2^{k(1-\delta)}2^{k\delta} \|\Delta_k u\|_{L^\infty}\|\widetilde\Delta_k\theta\|_{L^\infty} \nonumber\\
  &\leq& C \|u\|_{\dot C^{\delta}}\lp \sum_{k\geq j-3}2^{k(1-\delta)}\|\Delta_k\theta\|_{L^\infty}\rp.
\end{eqnarray}
Inserting the upper estimates (\ref{I1Linf})-(\ref{I3Linf}) into (\ref{thejLinf1}), we have
\begin{eqnarray}
  \frac{d}{dt}  \|\Delta_j\theta\|_{L^\infty}+c 2^{j(\alpha-\sigma)} \|\Delta_j\theta\|_{L^\infty}&\leq& C_1 \|\Delta_j \theta\|_{L^\infty}+
  \|\Delta_j f\|_{L^\infty}+C\|u\|_{\dot C^\delta} 2^{-j\delta} \sum_{k\leq j+4}2^k\| \Delta_k\theta\|_{L^\infty}+ \nonumber\\
  &&+\,C \|u\|_{\dot C^\delta} \sum_{k\geq j-3}2^{k(1-\delta)}\|\Delta_k\theta\|_{L^\infty} .
\end{eqnarray}
In particular, by some $j_0\in \N$ chosen later (cf. \eqref{j0}) so that $c 2^{j_0(\alpha-\sigma)} \geq  2 C_1$, or more precisely
\begin{equation}\label{j0-1}
  j_0\geq \Big[\frac{1}{\alpha-\sigma}\log_2\Big( \frac{2C_1}{c}\Big)\Big]+1,
\end{equation}
we see that for $j\geq j_0$,
\begin{eqnarray}\label{thejLinf2}
  \frac{d}{dt} \|\Delta_j\theta\|_{L^\infty}+ \frac{c}{2} 2^{j(\alpha-\sigma)} \|\Delta_j\theta\|_{L^\infty}
  &\leq& \|\Delta_j f\|_{L^\infty}+C\|u\|_{\dot C^\delta} 2^{-j\delta} \sum_{k\leq j+4}2^k\| \Delta_k\theta\|_{L^\infty}+ \nonumber\\
  &&+\,C \|u\|_{\dot C^\delta} \sum_{k\geq j-3}2^{k(1-\delta)}\|\Delta_k\theta\|_{L^\infty} \nonumber\\
  & := & F_j^1 + F_j^2 + F_j^3.
\end{eqnarray}
Consequently, Gr\"onwall's inequality guarantees that for every $j\geq j_0$ and $t\geq 0$,
\begin{eqnarray}\label{theLinf3}
  \|\Delta_j\theta(t)\|_{L^\infty} \leq e^{-\frac{c}{2}t 2^{j(\alpha-\sigma)}} \|\Delta_j\theta_0\|_{L^\infty}+ \int_0^t e^{-\frac{c}{2}(t-\tau)2^{j(\alpha-\sigma)}} \lp F_j^1(\tau) +F_j^2(\tau)+F_j^3(\tau)\rp\,\dd \tau.
\end{eqnarray}
On the other hand, we have the classical maximum principle \eqref{MP1} for the considered equation \eqref{DD}:
\begin{eqnarray}\label{MPes}
  \|\theta(t)\|_{L^\infty}\leq \|\theta_0\|_{ L^\infty}+\int_0^t \|f(\tau)\|_{ L^\infty}\dd t.
\end{eqnarray}
Observing that for all $t>0$, $j\in\N$ and $s\in(0,\alpha-\sigma)$,
\begin{eqnarray}\label{dataest}
  2^{j s}e^{-\frac{c}{2}t2^{j(\alpha-\sigma)}} \|\Delta_j\theta_0\|_{L^\infty} & \leq &   t^{-\frac{s}{\alpha-\sigma}}\lp (t\, 2^{j(\alpha-\sigma)})^{\frac{s}{\alpha-\sigma}}
  e^{-\frac{c}{2}t 2^{j(\alpha-\sigma)}} \rp \|\Delta_j\theta_0\|_{L^\infty} \nonumber \\
  & \leq & C_{\alpha,\sigma,s} t^{-\frac{s}{\alpha-\sigma}} \|\theta_0\|_{L^\infty},
\end{eqnarray}
we gather \eqref{theLinf3} and \eqref{MPes} to obtain
\begin{eqnarray}\label{key-est1}
  \|\theta(t)\|_{C^s} \approx \|\theta(t)\|_{ B^s_{\infty,\infty}}
  &\leq& \sup_{j\leq j_0} 2^{js}\|\Delta_j \theta(t)\|_{L^\infty} + \sup_{j\geq j_0}2^{js} \|\Delta_j\theta(t)\|_{L^\infty} \nonumber \\
  & \leq & C 2^{j_0} \lp \|\theta_0\|_{L^\infty}+ \|f\|_{L^1_t L^\infty} \rp + C_{\alpha,\sigma,s} t^{-\frac{s}{\alpha-\sigma}}\|\theta_0\|_{L^\infty} + \nonumber\\
  &&  + \sup_{j\geq j_0}\int_0^t e^{-\frac{c}{2}(t-\tau)2^{j(\alpha-\sigma)}} 2^{js} \lp F_j^1(\tau) + F_j^2(\tau) + F_j^3(\tau) \rp\dd \tau.
\end{eqnarray}
For the term containing $F^1_j$, we infer that for every $s\in (0,\alpha-\sigma+\delta)$,
\begin{eqnarray}\label{esFj1}
  \sup_{j\geq j_0}\int_0^t e^{-\frac{c}{2}(t-\tau)2^{j(\alpha-\sigma)}} 2^{js} F_j^1(\tau)\dd\tau & = & \sup_{j\geq j_0} \int_0^t e^{-\frac{c}{2}(t-\tau)2^{j(\alpha-\sigma)}}2^{js}\|\Delta_jf(\tau)\|_{L^\infty}\,\dd\tau \nonumber\\
  & \leq & C \sup_{j\geq j_0} \int_0^t e^{-\frac{c}{2} (t-\tau)2^{j(\alpha-\sigma)}} 2^{j(s-\delta)}\|f(\tau)\|_{\dot C^\delta}\dd\tau \nonumber\\
  & \leq & C  \|f\|_{L^\infty_t\dot C^\delta} \sup_{j\geq j_0} 2^{j(s-\delta)}
  \int_0^t e^{-\frac{c}{2} (t-\tau)2^{j(\alpha-\sigma)}} \,\dd\tau \nonumber \\
  & \leq & C \|f\|_{L^\infty_t\dot C^\delta} \sup_{j\geq j_0} 2^{j(s-\alpha +\sigma-\delta)} \leq C \|f\|_{L^\infty_t\dot C^\delta} .
\end{eqnarray}
For the term including $F^2_j$, thanks to \eqref{key-eq2} in Lemma \ref{lem-key}, we deduce that for every $s\in (0,\alpha-\sigma)$ and $\delta\in(1-\alpha+\sigma,1)$,
\begin{eqnarray}\label{esFj2}
  &&\sup_{j\geq j_0} \int_0^t e^{-\frac{c}{2}(t-\tau)2^{j(\alpha-\sigma)}} 2^{js} F_j^2(\tau)\dd\,\tau \nonumber \\
  &=& C \sup_{j\geq j_0} \int_0^t
  e^{-\frac{c}{2}(t-\tau)2^{j(\alpha-\sigma)}} \|u(\tau)\|_{\dot C^\delta} 2^{j(s-\delta)} \bigg(\sum_{k\leq j+4} 2^k \|\Delta_k \theta(\tau)\|_{L^\infty}\bigg) \dd\tau \nonumber\\
  &\leq &  C \|u\|_{L^\infty_t \dot C^\delta}\sup_{j\geq j_0} \int_0^t e^{-\frac{c}{2}(t-\tau)2^{j(\alpha-\sigma)}} 2^{j(s-\delta)} \bigg(\sum_{k\leq j+4} 2^{k(1-s)}\|\theta(\tau)\|_{B^s_{\infty,\infty}}\bigg)\dd\tau \nonumber\\
  &\leq & C \|u\|_{L^\infty \dot C^\delta} \lp \sup_{\tau\in (0,t]}\tau^{\frac{s}{\alpha-\sigma}}\|\theta(\tau)\|_{ B^s_{\infty,\infty}}\rp
  \sup_{j\geq j_0} 2^{j(1-\delta)} \int_0^t e^{-\frac{c}{2}(t-\tau)2^{j(\alpha-\sigma)}}\tau^{-\frac{s}{\alpha-\sigma}}\dd\tau \nonumber \\
  &\leq & C \|u\|_{L^\infty \dot C^\delta} \lp \sup_{\tau\in (0,t]}\tau^{\frac{s}{\alpha-\sigma}}\|\theta(\tau)\|_{ B^s_{\infty,\infty}}\rp
  t^{-\frac{s}{\alpha-\sigma}} \sup_{j\geq j_0} 2^{j(1-\delta-\alpha+\sigma)} \nonumber \\
  & \leq & C t^{-\frac{s}{\alpha-\sigma}}  2^{j_0(1-\alpha+\sigma-\delta)} \|u\|_{L^\infty_t \dot C^\delta}\lp \sup_{\tau\in (0,t]}\tau^{\frac{s}{\alpha-\sigma}}\|\theta(\tau)\|_{ B^s_{\infty,\infty}}\rp.
\end{eqnarray}
For the term including $F^3_j$ in \eqref{key-est1}, by using \eqref{key-eq2} again, we similarly get that for all $s \in (1-\delta,\alpha-\sigma)$ and $\delta\in(1-\alpha+\sigma,1)$,
\begin{eqnarray}\label{esFj3}
  &&\sup_{j\geq j_0} \int_0^t e^{-\frac{c}{2}(t-\tau)2^{j(\alpha-\sigma)}} 2^{js} F_j^3(\tau)\,\dd\tau \nonumber \\
  &=& C \sup_{j\geq j_0} \int_0^t
  e^{-\frac{c}{2}(t-\tau)2^{j(\alpha-\sigma)}} \|u(\tau)\|_{\dot C^\delta} 2^{js} \bigg(\sum_{k\geq j-3} 2^{k(1-\delta)} \|\Delta_k \theta(\tau)\|_{L^\infty}\bigg) \dd\tau \nonumber\\
  &\leq &  C \|u\|_{L^\infty_t \dot C^\delta}\sup_{j\geq j_0} 2^{js}\bigg(\sum_{k\geq j-3} 2^{k(1-\delta-s)}\bigg)
  \int_0^t e^{-\frac{c}{2}(t-\tau)2^{j(\alpha-\sigma)}} \|\theta(\tau)\|_{B^s_{\infty,\infty}}\dd\tau \nonumber\\
  &\leq & C \|u\|_{L^\infty \dot C^\delta} \lp \sup_{\tau\in (0,t]}\tau^{\frac{s}{\alpha-\sigma}}\|\theta(\tau)\|_{ B^s_{\infty,\infty}}\rp
  \sup_{j\geq j_0} 2^{j(1-\delta)} \int_0^t e^{-\frac{c}{2}(t-\tau)2^{j(\alpha-\sigma)}}\tau^{-\frac{s}{\alpha-\sigma}}\dd\tau \nonumber \\
  & \leq & C t^{-\frac{s}{\alpha-\sigma}}  2^{j_0(1-\alpha+\sigma-\delta)} \|u\|_{L^\infty_t \dot C^\delta}\lp \sup_{\tau\in (0,t]}\tau^{\frac{s}{\alpha-\sigma}}\|\theta(\tau)\|_{ B^s_{\infty,\infty}}\rp.
\end{eqnarray}
Inserting the above estimates (\ref{esFj1}), (\ref{esFj2}), (\ref{esFj3}) into (\ref{key-est1}) yields that for any $1-\delta<s<\alpha-\sigma$ and $0<t\leq T$,
\begin{eqnarray}\label{theCs}
  t^{\frac{s}{\alpha-\sigma}} \|\theta(t)\|_{B_{\infty,\infty}^s} & \leq &
  C t^{\frac{s}{\alpha-\sigma}} \lp \|\theta_0\|_{L^\infty} + \|f\|_{L^1_t L^\infty}\rp 2^{j_0} + C_{\alpha,\sigma,s} \|\theta_0\|_{L^\infty}+
  C t^{\frac{s}{\alpha-\sigma}}\|f\|_{L^\infty_t \dot C^\delta} + \nonumber\\
  && +\, C 2^{j_0(1-\alpha+\sigma-\delta)} \|u\|_{L^\infty_t \dot C^\delta}\lp \sup_{\tau\in (0,t]}\tau^{\frac{s}{\alpha-\sigma}}\|\theta(\tau)\|_{ B^s_{\infty,\infty}}\rp \nonumber \\
  & \leq &
  C T^{\frac{s}{\alpha-\sigma}} \lp \|\theta_0\|_{L^\infty} + \|f\|_{L^1_T L^\infty}\rp 2^{j_0} + C_{\alpha,\sigma,s} \|\theta_0\|_{L^\infty}+
  C T^{\frac{s}{\alpha-\sigma}}\|f\|_{L^\infty_T \dot C^\delta} + \nonumber\\
  && +\, C 2^{j_0(1-\alpha+\sigma-\delta)} \|u\|_{L^\infty_T \dot C^\delta}\lp \sup_{t\in (0,T]} t^{\frac{s}{\alpha-\sigma}}\|\theta(t)\|_{ B^s_{\infty,\infty}}\rp.
\end{eqnarray}
Since $1-\alpha+\sigma -\delta>0$, by further choosing $j_0$ such that $C 2^{j_0(1-\alpha+\sigma-\delta)} \|u\|_{L^\infty_T \dot C^\delta}\leq \frac{1}{2}$ and \eqref{j0-1} holds, or more precisely,
\begin{eqnarray}\label{j0}
  j_0 := \max\set{\Big[ \frac{1}{\delta-(1-\alpha+\sigma)}\log_2 \lp2C \|u\|_{L^\infty_T \dot C^\delta}\rp\Big], \Big[\frac{1}{\alpha-\sigma}\log_2\Big( \frac{2C_1}{c}\Big)\Big]}+1,
\end{eqnarray}
we have that for all $1-\delta<s<\alpha-\sigma$,
\begin{eqnarray}
  \sup_{t\in (0,T]} \lp t^{\frac{s}{\alpha-\sigma}}\|\theta(t)\|_{ B^s_{\infty,\infty}} \rp \leq  C( T + 1)\lp  2^{j_0}\big(\|\theta_0\|_{L^\infty} + \|f\|_{L^1_T L^\infty}\big) +
  \|f\|_{L^\infty_T\dot C^\delta} \rp,
\end{eqnarray}
which implies that for arbitrarily small $t_0\in (0,T)$ and every $s_0\in (1-\delta,\alpha-\sigma)$,
\begin{eqnarray}
  \sup_{t\in [t_0,T]} \|\theta(t)\|_{ B^{s_0}_{\infty,\infty}} \leq C t_0^{-\frac{s_0}{\alpha-\sigma}}( T + 1) \lp  2^{j_0}\big(\|\theta_0\|_{L^\infty} + \|f\|_{L^1_T L^\infty}\big) +
  \|f\|_{L^\infty_T\dot C^\delta} \rp,
\end{eqnarray}
with $j_0$ given by \eqref{j0}.

\textbf{Step 2:} the estimation of $\|\theta\|_{L^\infty([t_1,T]; B^{s_0+s_1}_{\infty,\infty})}$ for $s_0,s_1\in (1-\delta,\alpha-\sigma)$ and any $t_1\in (t_0,T)$.

For every $j\geq j_0$ with $j_0\in\N$ satisfying \eqref{j0-1} chosen later, applying the Gr\"onwall inequality to (\ref{thejLinf2}) over the time interval $[t_0, t]$ (for $t>t_0>0$) gives
\begin{equation}\label{theLinf4}
  \|\Delta_j\theta(t)\|_{L^\infty} \leq e^{-\frac{c}{2} (t-t_0) 2^{j(\alpha-\sigma)}} \|\Delta_j\theta(t_0)\|_{L^\infty}
  + \int_{t_0}^t e^{-\frac{c}{2}(t-\tau)2^{j(\alpha-\sigma)}} \lp F_j^1(\tau) +F_j^2(\tau)+F_j^3(\tau)\rp\,\dd \tau.
\end{equation}
Noticing that for $j\in \N$, $s_0\in(1-\delta,\alpha-\sigma)$ and all $s\in (0,\alpha-\sigma)$,
\begin{eqnarray}\label{dataest1}
  && e^{-\frac{c}{2}(t-t_0)2^{j(\alpha-\sigma)}}2^{j (s_0 + s)} \|\Delta_{j}\theta(t_{0})\|_{L^{\infty}}\nonumber\\
  &\leq& e^{-\frac{c}{2}(t-t_0)2^{j(\alpha-\sigma)}}2^{js} \|\theta(t_0)\|_{B_{\infty,\infty}^{s_0}} \nonumber\\
  &\leq & C(t-t_{0})^{-\frac{s}{\alpha-\sigma}}\Big((t-t_0)2^{j(\alpha-\sigma)}\Big)^{\frac{s}{\alpha-\sigma}} e^{-\frac{c}{2}(t-t_0)2^{j(\alpha-\sigma)}}
  \|\theta(t_0)\|_{B_{\infty,\infty}^{s_0}} \nonumber\\
  &\leq & C_{\alpha,\sigma,s} (t-t_0)^{-\frac{s}{\alpha-\sigma}}\|\theta(t_0)\|_{B_{\infty,\infty}^{s_0}},
\end{eqnarray}
by arguing as \eqref{key-est1} we obtain that for all $t\geq t_0>0$,
\begin{eqnarray}\label{key-est2}
  \|\theta(t)\|_{ B^{s_0+s}_{\infty,\infty}}
  &\leq& \sup_{j\leq j_0} 2^{j(s_0+s)}\|\Delta_j \theta(t)\|_{L^\infty} + \sup_{j\geq j_0}2^{j(s_0+s)} \|\Delta_j\theta(t)\|_{L^\infty} \nonumber \\
  & \leq & C 2^{j_0(s_0+s)} \lp \|\theta_0\|_{L^\infty}+ \|f\|_{L^1_t L^\infty} \rp + C_{\alpha,\sigma,s} (t-t_0)^{-\frac{s}{\alpha-\sigma}}\|\theta(t_0)\|_{B^{s_0}_{\infty,\infty}} + \nonumber\\
  &&  + \sup_{j\geq j_0}\int_{t_0}^t e^{-\frac{c}{2}(t-\tau) 2^{j(\alpha-\sigma)}} 2^{j(s_0+s)} \lp F_j^1(\tau) + F_j^2(\tau) + F_j^3(\tau) \rp\dd \tau.
\end{eqnarray}
For the term containing $F^1_j$, similarly as obtaining \eqref{esFj1}, we get that for every $s\in(0,\alpha-\sigma)$ and $s_0+s<\delta+\alpha-\sigma$,
\begin{eqnarray}\label{esFj1-2}
  && \sup_{j\geq j_0}\int_{t_0}^t e^{-\frac{c}{2}(t-\tau)2^{j(\alpha-\sigma)}} 2^{j(s_0+s)} F_j^1(\tau)\dd\tau \nonumber \\
  & = & \sup_{j\geq j_0} \int_{t_0}^t e^{-\frac{c}{2}(t-\tau)2^{j(\alpha-\sigma)}}2^{j(s_0+s)}\|\Delta_jf(\tau)\|_{L^\infty}\,\dd\tau \nonumber\\
  & \leq & C \sup_{j\geq j_0} \int_{t_0}^t e^{-\frac{c}{2} (t-\tau)2^{j(\alpha-\sigma)}} 2^{j(s_0+s-\delta)}\|f(\tau)\|_{\dot C^\delta}\dd\tau \nonumber\\
  & \leq & C \|f\|_{L^\infty_t \dot C^\delta} \sup_{j\geq j_0} 2^{j(s_0+s-\delta)}
  \int_{t_0}^t e^{-\frac{c}{2} (t-\tau)2^{j(\alpha-\sigma)}} \,\dd\tau \nonumber \\
  & \leq & C \|f\|_{L^\infty_t \dot C^\delta} \sup_{j\geq j_0} 2^{j(s_0+s-\alpha +\sigma-\delta)} \leq C \|f\|_{L^\infty_t \dot C^\delta}.
\end{eqnarray}
For the term including $F^2_j$ in \eqref{key-est2}, by arguing as \eqref{esFj2}, we deduce that for every $s\in (0,\alpha-\sigma)$ and $s_0+s\leq 1$,
\begin{eqnarray}\label{esFj2-2}
  &&\sup_{j\geq j_0} \int_{t_0}^t e^{-\frac{c}{2}(t-\tau)2^{j(\alpha-\sigma)}} 2^{j(s_0+s)} F_j^2(\tau)\,\dd\tau \nonumber \\
  &=& C \sup_{j\geq j_0} \int_{t_0}^t
  e^{-\frac{c}{2}(t-\tau)2^{j(\alpha-\sigma)}} \|u(\tau)\|_{\dot C^\delta} 2^{j(s_0+s-\delta)} \bigg(\sum_{-1\leq k\leq j+4} 2^k \|\Delta_k \theta(\tau)\|_{L^\infty}\bigg) \dd\tau \nonumber\\
  &\leq &  C \|u\|_{L^\infty_t \dot C^\delta}\sup_{j\geq j_0} \int_{t_0}^t e^{-\frac{c}{2}(t-\tau)2^{j(\alpha-\sigma)}} 2^{j(s_0+s-\delta)}
  \bigg(\sum_{-1\leq k\leq j+4} 2^{k(1-s-s_0)}\|\theta(\tau)\|_{B^{s_0+s}_{\infty,\infty}}\bigg)\dd\tau \nonumber\\
  &\leq & C \|u\|_{L^\infty_t \dot C^\delta} \lp \sup_{\tau\in (t_0,t]}(\tau-t_0)^{\frac{s}{\alpha-\sigma}}\|\theta(\tau)\|_{ B^{s_0+s}_{\infty,\infty}}\rp
  \sup_{j\geq j_0} \lp 2^{j(1-\delta)} j \rp \int_{t_0}^t e^{-\frac{c}{2}(t-\tau)2^{j(\alpha-\sigma)}}(\tau-t_0)^{-\frac{s}{\alpha-\sigma}}\dd\tau \nonumber \\
  &\leq & C \|u\|_{L^\infty_t \dot C^\delta} \lp \sup_{\tau\in (t_0,t]}(\tau-t_0)^{\frac{s}{\alpha-\sigma}}\|\theta(\tau)\|_{ B^s_{\infty,\infty}}\rp
  (t-t_0)^{-\frac{s}{\alpha-\sigma}} \sup_{j\geq j_0} \lp 2^{j(1-\delta-\alpha+\sigma)}j \rp \nonumber \\
  & \leq & C \|u\|_{L^\infty_t \dot C^\delta}\lp \sup_{\tau\in (t_0,t]}(\tau-t_0)^{\frac{s}{\alpha-\sigma}}\|\theta(\tau)\|_{ B^s_{\infty,\infty}}\rp
  (t-t_0)^{-\frac{s}{\alpha-\sigma}}  2^{j_0 \frac{1-\alpha+\sigma-\delta}{2}},
\end{eqnarray}
and for $1<s_0+s < \delta+\alpha-\sigma$,
\begin{eqnarray}\label{esFj2-22}
  &&\sup_{j\geq j_0} \int_{t_0}^t e^{-\frac{c}{2}(t-\tau)2^{j(\alpha-\sigma)}} 2^{j(s_0+s)} F_j^2(\tau)\,\dd\tau \nonumber \\
  &\leq &  C \|u\|_{L^\infty_t \dot C^\delta}\sup_{j\geq j_0} \int_{t_0}^t e^{-\frac{c}{2}(t-\tau)2^{j(\alpha-\sigma)}} 2^{j(s_0+s-\delta)}
  \bigg(\sum_{-1\leq k\leq j+4} 2^{k(1-s-s_0)}\|\theta(\tau)\|_{B^{s_0+s}_{\infty,\infty}}\bigg)\dd\tau \nonumber\\
  &\leq & C \|u\|_{L^\infty_t \dot C^\delta} \lp \sup_{\tau\in (t_0,t]}(\tau-t_0)^{\frac{s}{\alpha-\sigma}}\|\theta(\tau)\|_{ B^{s_0+s}_{\infty,\infty}}\rp
  \sup_{j\geq j_0}  2^{j(s_0+s-\delta)} \int_{t_0}^t e^{-\frac{c}{2}(t-\tau)2^{j(\alpha-\sigma)}}(\tau-t_0)^{-\frac{s}{\alpha-\sigma}}\dd\tau \nonumber \\
  &\leq & C \|u\|_{L^\infty_t \dot C^\delta} \lp \sup_{\tau\in (t_0,t]}(\tau-t_0)^{\frac{s}{\alpha-\sigma}}\|\theta(\tau)\|_{ B^{s_0+s}_{\infty,\infty}}\rp
  (t-t_0)^{-\frac{s}{\alpha-\sigma}} \sup_{j\geq j_0} \lp 2^{j(s_0+s-\delta-\alpha+\sigma)} \rp \nonumber \\
  & \leq & C \|u\|_{L^\infty_t \dot C^\delta}\lp \sup_{\tau\in (t_0,t]}(\tau-t_0)^{\frac{s}{\alpha-\sigma}}\|\theta(\tau)\|_{ B^{s_0+s}_{\infty,\infty}}\rp
  (t-t_0)^{-\frac{s}{\alpha-\sigma}}  2^{j_0 (s_0+s-\alpha+\sigma-\delta)}.
\end{eqnarray}
For the term including $F^3_j$ in \eqref{key-est2}, by using \eqref{key-eq2} again, we estimate similarly as \eqref{esFj3} to get that for all $s\in (1-\delta,\alpha-\sigma)$,
\begin{eqnarray}\label{esFj3-2}
  &&\sup_{j\geq j_0} \int_{t_0}^t e^{-\frac{c}{2}(t-\tau)2^{j(\alpha-\sigma)}} 2^{j(s_0+s)} F_j^3(\tau)\,\dd\tau \nonumber \\
  &=& C \sup_{j\geq j_0} \int_{t_0}^t
  e^{-\frac{c}{2}(t-\tau)2^{j(\alpha-\sigma)}} \|u(\tau)\|_{\dot C^\delta} 2^{j(s_0+s)} \bigg(\sum_{k\geq j-3} 2^{k(1-\delta)} \|\Delta_k \theta(\tau)\|_{L^\infty}\bigg) \dd\tau \nonumber\\
  &\leq &  C \|u\|_{L^\infty_t \dot C^\delta}\sup_{j\geq j_0} 2^{j(s_0+s)}\bigg(\sum_{k\geq j-3} 2^{k(1-\delta-s_0-s)}\bigg)
  \int_{t_0}^t e^{-\frac{c}{2}(t-\tau)2^{j(\alpha-\sigma)}} \|\theta(\tau)\|_{B^{s_0+s}_{\infty,\infty}}\dd\tau \nonumber\\
  &\leq & C \|u\|_{L^\infty_t \dot C^\delta} \lp \sup_{\tau\in (t_0,t]}(\tau-t_0)^{\frac{s}{\alpha-\sigma}}\|\theta(\tau)\|_{ B^{s_0+s}_{\infty,\infty}}\rp
  \sup_{j\geq j_0} 2^{j(1-\delta)} \int_{t_0}^t e^{-\frac{c}{2}(t-\tau)2^{j(\alpha-\sigma)}}(\tau-t_0)^{-\frac{s}{\alpha-\sigma}}\dd\tau \nonumber \\
  & \leq & C \|u\|_{L^\infty_t \dot C^\delta}\lp \sup_{\tau\in (t_0,t]}(\tau-t_0)^{\frac{s}{\alpha-\sigma}}\|\theta(\tau)\|_{ B^{s_0+s}_{\infty,\infty}}\rp
  (t-t_0)^{-\frac{s}{\alpha-\sigma}}  2^{j_0(1-\alpha+\sigma-\delta)} .
\end{eqnarray}
Plugging the estimates \eqref{esFj1-2}-\eqref{esFj3-2} into \eqref{key-est2}, and in a similar way as obtaining \eqref{theCs}, we have that for every $t\in (t_0,T]$,
$s\in (1-\delta,\alpha-\sigma)$ and $ s_0+s < \delta+\alpha-\sigma$,
\begin{eqnarray}
  && (t-t_0)^{\frac{s}{\alpha-\sigma}} \|\theta(t)\|_{B_{\infty,\infty}^{s_0+s}} \nonumber \\
  & \leq &
  C T^{\frac{s}{\alpha-\sigma}} \lp \|\theta_0\|_{L^\infty} + \|f\|_{L^1_T L^\infty}\rp 2^{j_0(s_0+s)} + C_{\alpha,\sigma,s} \|\theta(t_0)\|_{B^{s_0}_{\infty,\infty}}+
  C T^{\frac{s}{\alpha-\sigma}}\|f\|_{L^\infty_t\dot C^\delta} + \nonumber\\
  && + \,
  \begin{cases}
    C 2^{j_0\frac{1-\alpha+\sigma-\delta}{2}} \|u\|_{L^\infty_T \dot C^\delta} \lp \sup_{t\in (t_0,T]} (t-t_0)^{\frac{s}{\alpha-\sigma}}\|\theta(t)\|_{B^{s_0+s}_{\infty,\infty}}\rp,
    \quad &\textrm{if}\;\; s_0+s \leq 1, \\
    C 2^{j_0(s_0+s-\alpha+\sigma-\delta)} \|u\|_{L^\infty_T \dot C^\delta}\lp \sup_{t\in (t_0,T]} (t-t_0)^{\frac{s}{\alpha-\sigma}}\|\theta(t)\|_{B^{s_0+s}_{\infty,\infty}}\rp,
    \quad & \textrm{if}\;\; 1< s_0+s < \delta+\alpha-\sigma.
  \end{cases}
\end{eqnarray}
Hence by choosing $j_0\in \N$ as
\begin{eqnarray}\label{j0-2}
  j_0 :=
  \begin{cases}
    \;\max\set{\Big[\frac{2}{\delta-(1-\alpha+\sigma)}\log_2 \lp2C \|u\|_{L^\infty_T \dot C^\delta}\rp\Big], \Big[\frac{1}{\alpha-\sigma}\log_2\Big( \frac{2C_1}{c}\Big)\Big]}+1,
    \quad &\textrm{if}\;\; s_0+s \leq 1, \\
    \max\set{\Big[ \frac{1}{\delta +\alpha-\sigma-(s_0+s)}\log_2 \lp2C \|u\|_{L^\infty_T \dot C^\delta}\rp\Big], \Big[\frac{1}{\alpha-\sigma}\log_2\Big( \frac{2C_1}{c}\Big)\Big]}+1,
    \quad & \textrm{if}\;\; 1< s_0+s < \delta+\alpha-\sigma,
  \end{cases}
\end{eqnarray}
we find that for all $s\in (1-\delta,\alpha-\sigma)$ and $ s_0+s < \delta+\alpha-\sigma$,
\begin{eqnarray}
  && \sup_{t\in (t_0,T]} \lp (t-t_0)^{\frac{s}{\alpha-\sigma}}\|\theta(t)\|_{ B^{s_0+s}_{\infty,\infty}} \rp \nonumber\\
  & \leq &  C( T + 1) \lp \|\theta_0\|_{L^\infty}
  +\|f\|_{L^1_T L^\infty}\rp 2^{j_0(s_0+s)} + C\|\theta(t_0)\|_{B^{s_0}_{\infty,\infty}}  +  C(T+1)\|f\|_{L^\infty_T\dot C^\delta},
\end{eqnarray}
which specially guarantees that for any $t_1>t_0>0$ (which may be arbitrarily close to $t_0$) and every
$s_0, s_1 \in ( 1-\delta, \alpha-\sigma)$ satisfying $s_0+s_1<\delta+\alpha-\sigma$,
\begin{eqnarray}\label{esthe-s0s1}
  \sup_{t\in [t_1,T]} \|\theta(t)\|_{ B^{s_0+s_1}_{\infty,\infty}} & \leq & C (t_1-t_0)^{-\frac{s_1}{\alpha-\sigma}}\lp ( T + 1) \big(\|\theta_0\|_{L^\infty} + \|f\|_{L^1_T L^\infty}\big)2^{j_0(s_0+s_1)} + \|\theta(t_0)\|_{B^{s_0}_{\infty,\infty}}\rp + \nonumber\\
  && + \, C(t_1-t_0)^{-\frac{s_1}{\alpha-\sigma}}(T+1)\|f\|_{L^\infty_T \dot C^\delta},
\end{eqnarray}
with $j_0$ given by \eqref{j0-2}.

\textbf{Step 3: }the estimation of $\|\theta\|_{L^\infty([\tilde{t}, T];C^{1,\gamma})}$ for some $\gamma>0$ and any $\tilde{t}\in (0,T)$.

If $\alpha-\sigma \in (\frac{1}{2},1)$, we can select appropriate $s_0,s_1\in(1-\delta,\alpha-\sigma)$ so that $1<s_0+s_1<\delta+\alpha-\sigma$, thus from \eqref{esthe-s0s1}
we obtain that for $\gamma=s_0+s_1-1>0$,
\begin{eqnarray*}
  \sup_{t\in [t_1,T]}\|\theta(t)\|_{C^{1,\gamma}}\approx \sup_{t\in [t_1,T]} \|\theta(t)\|_{B^{s_0+s_1}_{\infty,\infty}}\leq C,
\end{eqnarray*}
with $C$ the bound on the right-hand-side of \eqref{esthe-s0s1}.

For the remained scope $\alpha-\sigma\in (0,\frac{1}{2}]$, we have to iterate the above procedure in \textbf{Step 2} for more times.
Assume that for some small number $t_k>0$, $k\in\N$, we have a finite bound on $\|\theta(t_k)\|_{B^{s_0+s_1+\cdots+s_k}_{\infty,\infty}}$ with
$s_0,s_1,\cdots,s_k\in (1-\delta,\alpha-\sigma)$ satisfying $s_0+s_1+\cdots+s_k\leq 1$,
then by arguing as \eqref{esthe-s0s1}, we infer that for any $t_{k+1}>t_k$, $s_{k+1}\in (1-\delta,\alpha-\sigma)$ satisfying $s_0+s_1+\cdots+s_{k+1}<\delta+\alpha-\sigma$,
\begin{eqnarray}\label{esthe-s0sk}
  && \sup_{t\in [t_{k+1},T]} \|\theta(t)\|_{ B^{s_0+s_1+\cdots+s_{k+1}}_{\infty,\infty}} \nonumber \\
  & \leq & C (t_{k+1}-t_k)^{-\frac{s_{k+1}}{\alpha-\sigma}}\lp ( T + 1) \big( \|\theta_0\|_{L^\infty} + \|f\|_{L^1_T L^\infty}\big) 2^{j_0(s_0+s_1+\cdots+s_{k+1})}  + \|\theta(t_k)\|_{B^{s_0+s_1+\cdots+s_k}_{\infty,\infty}}\rp + \nonumber\\
  && + \, C(t_{k+1}-t_k)^{-\frac{s_{k+1}}{\alpha-\sigma}}(T+1)\|f\|_{L^\infty_T \dot C^\delta},
\end{eqnarray}
where $j_0$ is also given by \eqref{j0-2} with $s_0+s_1$ replaced by $s_0+s_1+\cdots+s_{k+1}$. Hence if $\alpha-\sigma\in (\frac{1}{k+2},\frac{1}{k+1}]$, $k\in\N^+$,
we can choose appropriate numbers $s_0,s_1,\cdots,s_{k+1}\in (1-\delta,\alpha-\sigma)$ so that $1< s_0+s_1+ \cdots+ s_{k+1} <\delta +\alpha-\sigma$,
and by repeating the above process for $(k+1)$-times, we deduce that for $\gamma=s_0+s_1+\cdots+ s_{k+1}-1>0$,
\begin{eqnarray}\label{estheC1+}
  \sup_{t\in [t_{k+1},T]}\|\theta(t)\|_{C^{1,\gamma}}\approx \sup_{t\in [t_{k+1},T]} \|\theta(t)\|_{B^{s_0+s_1+\cdots+s_{k+1}}_{\infty,\infty}}\leq C
  \big( \|\theta_0\|_{L^\infty} + \|f\|_{L^\infty_T C^\delta}\big),
\end{eqnarray}
with $C$ a finite constant depending on $t_{k+1}, t_k,\cdots,t_0$, $s_{k+1},s_k,\cdots,s_0$, $\alpha$, $\sigma$, $\delta$, $T$, $d$ and $\|u\|_{L^\infty_T \dot C^\delta}$.

Therefore, for every $\alpha\in (0,1]$, $\sigma\in [0,\alpha)$, and for any $\tilde{t}\in (0,T)$, there is some $k\in \N$ so that $\alpha-\sigma\in (\frac{1}{k+2},\frac{1}{k+1}]$,
and we can choose $t_i= \frac{i+1}{k+2} \tilde{t}$ for $i=0,1,\cdots,k+1$ and appropriate numbers $s_0,s_1,\cdots,s_{k+1}\in (1-\delta,\alpha-\sigma)$ such that
$1<s_0+s_1+\cdots+s_{k+1}<\delta+\alpha-\sigma$, thus from \eqref{estheC1+} we conclude the key a priori estimate \eqref{target1}.

\subsection{The existence issue}\label{subsec:ext1}

We consider the following approximate system
\begin{equation}\label{appDD}
\begin{cases}
  \partial_t \theta + (u_\epsilon\cdot\nabla)\theta + \LL \theta -\epsilon \Delta \theta = f_\epsilon, \\
  u_\epsilon:= \phi_\epsilon* u, \quad f_\epsilon:= \phi_\epsilon* \lp f \,1_{B_{1/\epsilon}(0)} \rp,\\
  \theta|_{t=0}= \theta_{0,\epsilon}:=\phi_\epsilon* \lp \theta_0\, 1_{B_{1/\epsilon}(0)} \rp.
\end{cases}
\end{equation}
Here $1_{X}(x)$ is the standard indicator function on the set $X$, $\phi_\epsilon(x)=\epsilon^{-d} \phi(\frac{x}{\epsilon})$ for all $x\in \R^d$, and $\phi\in C^\infty_c(\R^d)$ is a test function supported on the ball $B_1(0)$ satisfying $\phi\equiv 1$ on $B_{1/2}(0)$ and $0\leq \phi\leq 1$.

Due to $\theta_0\in C_0(\R^d)$, we see that $\theta_{0,\epsilon}=\phi_\epsilon*\lp \theta_0 1_{B_{1/\epsilon}(0)}\rp$ is smooth defined for every $\epsilon>0$,
and $\|\theta_{0,\epsilon}\|_{H^s(\R^d)}\lesssim_\epsilon \|\theta_0\|_{L^\infty(\R^d)}$ for all $s\geq 0$. Similarly from $u\in L^\infty([0,T]; \dot C^\delta(\R^d))$ and $f\in L^\infty([0,T]; C^\delta (\R^d))$,
we get $u_\epsilon\in L^\infty([0,T];\dot C^s(\R^d))$ for all $s\geq \delta$ and $f_\epsilon\in L^\infty([0,T]; H^s(\R^d))$ for all $s\geq 0$.
Hence, for every $\epsilon>0$, by the classical method (e.g. cf. \cite[Proposition 7.1]{MX15}), we obtain an approximate solution
$\theta_\epsilon\in C([0,T]; H^s(\R^d))\cap C^\infty((0,T]\times\R^d)$, $s>\frac{d}{2}+1$ for the system \eqref{appDD}.

Since we have the following uniform-in-$\epsilon$ estimates that $\|\theta_{0,\epsilon}\|_{L^\infty}\leq \|\theta_0\|_{L^\infty}$,
$\|u_\epsilon\|_{L^\infty_T \dot C^\delta}\leq \|u\|_{L^\infty_T \dot C^\delta}$ and $\|f_\epsilon\|_{L^\infty_T C^\delta}\leq \|f\|_{L^\infty_T C^\delta}$,
we can consider the equation of $\theta_\epsilon$ and by arguing as \eqref{estheC1+} in the above subsection, we derive the uniform-in-$\epsilon$ estimate of
$\|\theta_\epsilon\|_{L^\infty((0,T]; C^{1,\gamma}(\mathbb{R}^d))}$ with some $\gamma>0$.
Such a uniform estimate guarantees that up to a subsequence, $\theta_\epsilon$ pointwisely converges to a function $\theta$ on $(0,T]\times \R^d$,
and also $\theta\in L^\infty((0,T];C^{1,\gamma}(\R^d))$ which satisfies \eqref{target1}. By passing $\epsilon$ to $0$ in \eqref{appDD}, we can see that $\theta$ is a distributional solution of \eqref{DD}.

\vskip .3in
\section{Proof of Theorem \ref{thm2}}\label{sec:thm2}
\setcounter{equation}{0}
Our main target of this section is to prove Theorem \ref{thm2}.

\subsection{Auxiliary lemmas}\label{subsec:thm2-1}

In this section we introduce some useful auxiliary lemmas.

The following lemma is concerned with the pointwise lower bound estimate of the Fourier symbol of the operator $\LL$.
\begin{lemma}\label{lem:A}
Let the diffusion operator $\LL$ be defined by \EQ{L} with the kernel function $K(y)=K(-y)$ satisfying \EQ{Kcd1}-\EQ{Kcd2},
then the associated symbol $A(\xi)$ given by \EQ{LKf} satisfies that
\begin{eqnarray}\label{Aest1}
  A(\xi)\geq  C^{-1} |\xi|^{\alpha-\sigma} - C,
\end{eqnarray}
where $\alpha\in ]0,1]$, $\sigma\in [0,\alpha[$ and $C=C(d,\alpha,\sigma)$ is a positive constant.
\end{lemma}

\begin{proof}[Proof of Lemma \ref{lem:A}]
  Recalling that one has (cf. Eq. (3.219) of \cite{Jacob})
\beqs
  |\xi|^\alpha = c_{d,\alpha}\;   \mathrm{p.v.}\int_{\R^d}\lp 1-\cos(y\cdot \xi)\rp \frac{1}{|y|^{d+\alpha}} \dd y,\quad \forall\alpha\in ]0,2[,
\eeqs
and by virtue of \EQ{Kcd1}-\EQ{Kcd2}, we get
\beqs
\begin{split}
  A(\xi) &  =  \,\textrm{p.v.}\int_{\R^d}\lp 1-\cos(y\cdot\xi)\rp K(y)\dd y \\
  & \geq c_2^{-1}  \int_{0<|y|\leq 1} \lp 1-\cos(y\cdot \xi)\rp \frac{1}{|y|^{d+(\alpha-\sigma)}}\dd y -  \int_{|y|\geq 1} |K(y)|\dd y \\
  & \geq c_2^{-1}  \lp  c_{d,\alpha}^{-1} |\xi|^{\alpha-\sigma}
  - \int_{|y|\geq 1}\frac{1}{|y|^{d+\alpha-\sigma}}\dd y\rp - c_1 \\
  & \geq c_2^{-1} c_{d,\alpha}^{-1} |\xi|^{\alpha-\sigma}  - C_{d,\alpha,\sigma}-c_1,
\end{split}
\eeqs
which corresponds to \EQ{Aest1}.
\end{proof}

With Lemma \ref{lem:A} in hand, we shall derive the following lower bound of some quantities involving the L\'evy-type operator $\LL$.
\begin{lemma}\label{lem:Lp1}
Let $p\geq 2$ and the symmetric kernel function $K(y)=K(-y)$ satisfy the conditions \eqref{Kcd1}-\eqref{Kcd2},
then for every $\theta\in \mathcal{S}(\R^d)$, we have
\begin{eqnarray}\label{Lp-keyes1}
  \int_{\mathbb{R}^d} |\theta(x)|^{p-2}\theta(x)\, \mathcal{L}\theta(x)\,\dd x \geq C \int_{\mathbb{R}^d} \Big(|D|^{\frac{\alpha-\sigma}{2}}|\theta(x)|^{\frac{p}{2}}\Big)^2
  \,\dd x- \widetilde{C}\int_{\mathbb{R}^d} |\theta(x)|^p\,\dd x,
\end{eqnarray}
and for every $j\in \N$,
\begin{eqnarray}\label{Lp-keyes2}
  \int_{\mathbb{R}^d} \mathcal{L}(\Delta_j\theta)\, (|\Delta_j\theta|^{p-2}\Delta_j\theta)\,\dd x
  \geq c 2^{j(\alpha-\sigma)}\|\Delta_j\theta\|_{L^p}^p- \widetilde{C} \|\Delta_j\theta\|_{L^p}^p,
\end{eqnarray}
where the constants $c,C>0$, $\widetilde{C}\geq 0$ depend only on the coefficients $p,\alpha,\sigma,d$.
\end{lemma}
\begin{proof}[Proof of Lemma \ref{lem:Lp1}]
  First we claim that the following estimate holds true
\begin{eqnarray}\label{eq:claim1}
  |\theta(x)|^{\frac{p}{2}-2}\theta(x)\,\mathcal{L}\theta(x) \geq  \frac{2}{p} (\mathcal{L}|\theta|^{\frac{p}{2}})(x) -
  2\int_{|x-y|\geq 1} \lp |\theta(x)|^{\frac{p}{2}} + |\theta(y)|^{\frac{p}{2}} \rp |K(x-y)|\,\dd y.
\end{eqnarray}
Indeed, according to the formula of $\LL$ \eqref{L} and the following estimate deduced from Young's inequality
\begin{eqnarray}\label{eq:fact}
  |\theta(x)|^{\frac{p}{2}-2}\theta(x)\theta(y)\leq|\theta(x)|^{\frac{p}{2}-1}|\theta(y)|\leq \frac{p-2}{p}|\theta(x)|^{\frac{p}{2}}+\frac{2}{p}|\theta(y)|^{\frac{p}{2}},
\end{eqnarray}
we have
\begin{eqnarray}
  |\theta(x)|^{\frac{p}{2}-2}\theta(x)\mathcal{L}\theta(x)
  &=&\textrm{p.v.}\int_{\mathbb{R}^d} \lp |\theta(x)|^{\frac{p}{2}}-|\theta(x)|^{\frac{p}{2}-2}
  \theta(x) \theta(y) \rp K(x-y)\,\dd y\nonumber\\
  &=& \textrm{p.v.}\int_{|x-y|\leq 1} \lp |\theta(x)|^{\frac{p}{2}}-|\theta(x)|^{\frac{p}{2}-2}
  \theta(x)\theta(y)\rp K(x-y)\,\dd y\nonumber\\
  &&+\int_{|x-y|\geq 1} \lp |\theta(x)|^{\frac{p}{2}}-|\theta(x)|^{\frac{p}{2}-2}\theta(x)\theta(y)\rp K(x-y)\,\dd y \\
  &\geq & \textrm{p.v.}\int_{|x-y|\leq 1} \lp |\theta(x)|^{\frac{p}{2}}-|\theta(x)|^{\frac{p}{2}-2}
  \theta(x)\theta(y)\rp K(x-y)\,\dd y\nonumber\\
  && - \frac{2p-2}{p} \int_{|x-y|\geq 1} \lp |\theta(x)|^{\frac{p}{2}}+ |\theta(y)|^{\frac{p}{2}}\rp |K(x-y)|\,\dd y.
\end{eqnarray}
Due to the positivity property of $K(y)$ on $0<|y|\leq 1$ and the inequality \eqref{eq:fact} again, we see that
\begin{eqnarray}
  &&\textrm{p.v.}\int_{|x-y|\leq 1} \lp |\theta(x)|^{\frac{p}{2}}-|\theta(x)|^{\frac{p}{2}-2}\theta(x)\theta(y)\rp K(x-y)\,\dd y\nonumber\\
  &\geq& \textrm{p.v.}\int_{|x-y|\leq 1} \left(|\theta(x)|^{\frac{p}{2}}-\Big(\frac{p-2}{p}|\theta(x)|^{\frac{p}{2}}+\frac{2}{p}|\theta(y)|^{\frac{p}{2}} \Big)\right)K(x-y)\,\dd y\nonumber\\
  &= & \frac{2}{p}\, \textrm{p.v.}\int_{|x-y|\leq 1} \left(|\theta(x)|^{\frac{p}{2}}-|\theta(y)|^{\frac{p}{2}} \right) K(x-y)\,\dd y \nonumber\\
  &=& \frac{2}{p} \,(\mathcal{L}|\theta|^{\frac{p}{2}})(x) - \frac{2}{p}\int_{|x-y|\geq1}
  \left(|\theta(x)|^{\frac{p}{2}}-|\theta(y)|^{\frac{p}{2}} \right)K(x-y)\,\dd y\nonumber \\
  & \geq & \frac{2}{p} \,(\mathcal{L}|\theta|^{\frac{p}{2}})(x) - \frac{2}{p}\int_{|x-y|\geq1}
  \left(|\theta(x)|^{\frac{p}{2}} + |\theta(y)|^{\frac{p}{2}} \right) |K(x-y)|\,\dd y .
\end{eqnarray}
Gathering the above estimates leads to \eqref{eq:claim1}.

As a consequence of \eqref{eq:claim1}, we get
\begin{eqnarray}\label{N1N2}
  \int_{\mathbb{R}^d} |\theta(x)|^{p-2}\theta(x) \mathcal{L}\theta(x)\,\dd x
  &=& \int_{\mathbb{R}^d} |\theta(x)|^{\frac{p}{2}} |\theta(x)|^{\frac{p}{2}-2}\theta(x)\mathcal{L}\theta(x)\,\dd x \nonumber\\
  &\geq& \frac{2}{p} \int_{\mathbb{R}^d} |\theta(x)|^{\frac{p}{2}}(\mathcal{L}|\theta|^{\frac{p}{2}})(x)\,\dd x\nonumber\\
  & & - 2 \int_{\mathbb{R}^d} |\theta(x)|^{\frac{p}{2}} \int_{|x-y|\geq 1} \big(|\theta(x)|^{\frac{p}{2}}+ |\theta(y)|^{\frac{p}{2}}\big) |K(x-y)|\,\dd y\dd x \nonumber\\
  &:=& N_1+ N_2.
\end{eqnarray}
In view of the Plancherel theorem and the estimate (\ref{Aest1}) concerning the symbol of $\LL$, it leads to
\begin{eqnarray*}
  N_1 &=& \frac{2}{p} \int_{\mathbb{R}^d} \widehat{|\theta|^{\frac{p}{2}}}(\xi) A(\xi) \widehat{|\theta|^{\frac{p}{2}}}(\xi)\,\dd \xi \nonumber\\
  &\geq & \frac{2}{p} C_{\alpha,\sigma,d}^{-1} \int_{\mathbb{R}^d} |\xi|^{\alpha-\sigma}\widehat{|\theta|^{\frac{p}{2}}}(\xi) \widehat{|\theta|^{\frac{p}{2}}}(\xi)\,\dd \xi
  - \frac{2}{p} C_{\alpha,\sigma,d}\int_{\mathbb{R}^d} \widehat{|\theta|^{\frac{p}{2}}}(\xi) \widehat{|\theta|^{\frac{p}{2}}}(\xi)\,\dd\xi\nonumber\\
  &=&\frac{2}{p} C_{\alpha,\sigma,d}^{-1} \int_{\mathbb{R}^d} \Big(|\xi|^{\frac{\alpha-\sigma}{2}}\widehat{|\theta|^{\frac{p}{2}}}(\xi)\Big)^2\,\dd\xi -
  \frac{2}{p} C_{\alpha,\sigma,d}\int_{\mathbb{R}^d} \widehat{|\theta|^{\frac{p}{2}}}(\xi) \widehat{|\theta|^{\frac{p}{2}}}(\xi)\,\dd\xi\nonumber\\
  &=& \frac{2}{p} C_{\alpha,\sigma,d}^{-1} \int_{\mathbb{R}^d}\Big(|D|^{\frac{\alpha-\sigma}{2}} |\theta(x)|^{\frac{p}{2}}\Big)^2 \,\dd x-
  \frac{2}{p} C_{\alpha,\sigma,d}\int_{\mathbb{R}^d} |\theta(x)|^p\,\dd x.
\end{eqnarray*}
The Young's inequality and the condition (\ref{Kcd1}) ensure that
\begin{eqnarray*}
  -\frac{N_2}{2}  &\leq& \int_{\mathbb{R}^d} |\theta(x)|^{\frac{p}{2}}\int_{|x-y|\geq 1} |\theta(x)|^{\frac{p}{2}} |K(x-y)|\,\dd y\,\dd x
  + \int_{\mathbb{R}^d} |\theta(x)|^{\frac{p}{2}}\int_{|x-y|\geq 1} |\theta(y)|^{\frac{p}{2}}\,|K(x-y)|\,\dd y\,\dd x\nonumber\\
  &\leq & \int_{\mathbb{R}^d} |\theta(x)|^p\int_{|x-y|\geq 1} |K(x-y)|\,\dd y \,\dd x
  + \|\theta\|_{L^p}^{\frac{p}{2}} \left\|\int_{\mathbb{R}^d} |\theta(y)|^{\frac{p}{2}}\,|K(x-y)|1_{\{|x-y|\geq 1\}}\,\dd y \right\|_{L^2_x} \nonumber\\
  &\leq & \|\theta\|_{L^p}^p \int_{|x|\geq 1}|K(x)|\,\dd x + \|\theta\|_{L^p}^{\frac{p}{2}} \||\theta(x)|^{\frac{p}{2}}\|_{L^2_x}  \int_{|x|\geq 1}|K(x)|\,\dd x \nonumber\\
  &\leq& 2c_1 \|\theta\|_{L^p}^p.
\end{eqnarray*}
Inserting the estimates of $N_1$ and $N_2$ into (\ref{N1N2}) yields the desired estimate (\ref{Lp-keyes1}).
Recalling the following inequality (cf. \cite{ChenMZ}) that
\begin{eqnarray*}
  \||D|^\beta(|\Delta_j \theta|^{\frac{p}{2}}) \|_{L^2}^2\geq \tilde{c} 2^{j\beta}\|\Delta_j\theta\|_{L^p}^p, \quad\textrm{for every   } \beta\in(0,2],p\in [2,\infty),j\in\N,
\end{eqnarray*}
with a constant $\tilde{c}>0$ independent of $j$, then the estimate \eqref{Lp-keyes2} is followed by combining the above lower bound with \eqref{Lp-keyes1}.
We thus conclude Lemma \ref{lem:Lp1}.
\end{proof}

Now we can show the key a priori $L^p$-estimate.
\begin{lemma}\label{lem:Lpes}
Let $u$ be a smooth vector field and $f$ be a smooth forcing term.
Assume that $\theta$ is a smooth solution for the drift-diffusion \eqref{DD}-\eqref{L} with $\theta_0\in L^p(\R^d)$ under the assumptions of $K$ \eqref{Kcd1}-\eqref{Kcd2}.
In addition, suppose that $u$ is divergence free. Then for any $T>0$, we have
\begin{eqnarray}\label{Lp-est1}
  \max_{0\leq t\leq T}\|\theta(t)\|_{L^{p}}\leq e^{C' T} \lp \|\theta_0\|_{ L^p} + \int_0^T \|f(t)\|_{L^p}\dd t\rp,
\end{eqnarray}
with $C'\geq 0$ depending only on $p,\alpha,\sigma,d$.
\end{lemma}
\begin{proof}[Proof of Lemma \ref{lem:Lpes}]
Multiplying both sided of (\ref{DD}) by $|\theta|^{p-2}\theta(x)$ and integrating over the spatial variable, we use the divergence-free condition of $u$ and H\"older's inequality to get
\begin{eqnarray}
  \frac{1}{p}\frac{d}{dt}\|\theta\|_{L^p}^p + \int_{\mathbb{R}^d} \mathcal{L}\theta(x)\, (|\theta|^{p-2}\theta)(x)\,\dd x
  \leq \|f\|_{L^p} \|\theta\|_{L^p}^{p-1}.\nonumber
\end{eqnarray}
Thanks to the following inequality (i.e., \eqref{Lp-keyes1})
\begin{eqnarray*}
  \int_{\mathbb{R}^d} \mathcal{L}\theta (|\theta|^{p-2}\theta)\,\dd x \geq
  C\int_{\mathbb{R}^d} \Big(|D|^{\frac{\alpha-\sigma}{2}} |\theta(x)|^{\frac{p}{2}}\Big)^2 \,\dd x- C'\int_{\mathbb{R}^d} |\theta(x)|^p\,\dd x,
\end{eqnarray*}
we obtain
\begin{eqnarray}
  \frac{1}{p}\frac{d}{dt}\|\theta\|_{L^{p}}^p + C \int_{\mathbb{R}^d} \Big(|D|^{\frac{\alpha-\sigma}{2}}|\theta(x)|^{\frac{p}{2}} \Big)^2 \,\dd x
  - C'\int_{\mathbb{R}^d} |\theta(x)|^p\,\dd x\leq \|f\|_{L^{p}}\|\theta\|_{L^p}^{p-1},\nonumber
\end{eqnarray}
which directly implies
\begin{eqnarray}
  \frac{d}{dt}\|\theta(t)\|_{L^p}- C'\|\theta(t)\|_{L^p}\leq \|f(t)\|_{L^p}.\nonumber
\end{eqnarray}
Hence Gr\"onwall's inequality guarantees the wanted inequality \eqref{Lp-est1}.
\end{proof}

\subsection{A priori estimates}

In this subsection, we assume $\theta$ is a smooth solution for the drift-diffusion equations \eqref{DD}-\eqref{L} with smooth $u$ and smooth $f$.
We shall show the estimate \eqref{target2} and the proof consists of three steps.

\textbf{Step 1:} the estimation of $\|\theta\|_{L^\infty([t_0,T];B_{p,\infty}^{s_0})}$ for any $s_0\in (0,\alpha-\sigma)$ and $t_0\in (0,T)$.

By applying the dyadic operator $\Delta_j$ ($j\in\N$) to the equation of $\theta$ in \eqref{DD}, similarly as \eqref{flEq1} and \eqref{Idec}, we get
\begin{eqnarray}\label{flEq2}
  \partial_t \Delta_j\theta+ u\cdot\nabla \Delta_j\theta +\mathcal{L}\Delta_j\theta &= &u\cdot\nabla \Delta_j\theta -\Delta_j(u\cdot\nabla \theta)+\Delta_j f \nonumber \\
  &=& I_1+I_2+I_3 + \Delta_j f,
\end{eqnarray}
where $I_1$-$I_3$ defined by \eqref{Idec} are the Bony's decomposition of the term $u\cdot\nabla \Delta_j\theta -\Delta_j(u\cdot\nabla \theta)$.
Multiplying both sides of the equation \eqref{flEq2} with $|\Delta_j \theta|^{p-2}\Delta_j \theta$ and integrating on the spatial variable over $\mathbb{R}^d$,
we use the divergence-free property of $u$ and the H\"older inequality to get
\begin{eqnarray}\label{eq:Lpest1}
  \frac{1}{p}\frac{d}{dt}\|\Delta_j\theta\|_{L^p}^p + \int_{\mathbb{R}^d} \mathcal{L}(\Delta_j\theta)(|\Delta_j\theta|^{p-2}\Delta_j\theta)\dd x
  \leq \lp \|\Delta_jf\|_{L^p} + \|I_1\|_{L^p} + \|I_2\|_{L^p} + \|I_3\|_{L^p} \rp\|\Delta_j\theta\|_{L^p}^{p-1}.
\end{eqnarray}
According to \eqref{Lp-keyes2} in Lemma \ref{lem:Lp1}, we see that
\begin{eqnarray}\label{eq:dissEs}
  \int_{\mathbb{R}^d} \mathcal{L}(\Delta_j\theta)\, (|\Delta_j\theta|^{p-2}\Delta_j\theta)\,\dd x
  \geq c 2^{j(\alpha-\sigma)}\|\Delta_j\theta\|_{L^p}^p- \widetilde{C} \|\Delta_j\theta\|_{L^p}^p,
\end{eqnarray}
where $c$ and $\widetilde{C}$ are constants depending on $p,\alpha,\sigma,d$.
Inserting \eqref{eq:dissEs} into \eqref{eq:Lpest1} and dividing $\|\Delta_j\theta\|_{L^p}^{p-1}$ lead to
\begin{eqnarray}
  \frac{d}{dt}\|\Delta_j\theta\|_{L^p} + c 2^{j(\alpha-\sigma)}\|\Delta_j \theta\|_{L^p} \leq
  \widetilde{C}\|\Delta_j\theta\|_{L^p} + \|\Delta_j f\|_{L^p} + \|I_1\|_{L^p} + \|I_2\|_{L^p} + \|I_3\|_{L^p}.
\end{eqnarray}
Similarly as deriving \eqref{I1Linf} and \eqref{I2Linf}, and using the following estimate on $\|I_3\|_{L^p}$ (from the divergence-free property of $u$):
\begin{eqnarray*}
  \|I_3\|_{L^p} &\leq& \sum_{k\geq j-2}\Big\|\nabla\cdot\Delta_j(\Delta_k u\, \widetilde\Delta_k\theta)\Big\|_{L^p}
  +  \sum_{k\geq j-2}\Big\|  \Delta_k u\cdot\nabla \widetilde\Delta_k\Delta_j\theta\Big\|_{L^p}\nonumber\\
  &\leq& C 2^j \sum_{k\geq j-2}2^j \|\Delta_k u\|_{L^\infty}\|\widetilde\Delta_k\theta\|_{L^p} \nonumber\\
  &\leq& C  \sum_{k\geq j-2}2^{k\delta} \|\Delta_k u\|_{L^\infty}2^{-k\delta}\|\widetilde\Delta_k\theta\|_{L^p} \nonumber\\
  &\leq& C \|u\|_{\dot C^{\delta}} \,2^j\Big( \sum_{k\geq j-2}2^{-k\delta}\|\Delta_k\theta\|_{L^p}\Big),
\end{eqnarray*}
we get
\begin{eqnarray*}
  \frac{d}{dt} \|\Delta_j\theta\|_{L^p}+c 2^{j(\alpha-\sigma)} \|\Delta_j\theta\|_{L^p}&\leq& \widetilde{C} \|\Delta_j \theta\|_{L^p}+
  \|\Delta_j f\|_{L^p}+C\|u\|_{\dot C^\delta} 2^{-j\delta} \sum_{k\leq j+4}2^k\| \Delta_k\theta\|_{L^p}+ \nonumber\\
  &&+\,C \|u\|_{\dot C^\delta} 2^j \sum_{k\geq j-3}2^{-k\delta}\|\Delta_k\theta\|_{L^p} .
\end{eqnarray*}
Let $j_0\in \N$ be a number chosen later (cf. \eqref{j0p}) which satisfies that $\frac{c}{2}2^{j_0(\alpha-\sigma)}\geq \widetilde{C}$, or more precisely,
\begin{equation}\label{j0-p1}
  j_0\geq \Big[\frac{1}{\alpha-\sigma}\log_2\Big( \frac{2\widetilde{C}}{c}\Big)\Big]+1,
\end{equation}
we infer that for all $j\geq j_0$,
\begin{eqnarray}\label{thejLp2}
  \frac{d}{dt} \|\Delta_j\theta\|_{L^p}+ \frac{c}{2} 2^{j(\alpha-\sigma)} \|\Delta_j\theta\|_{L^p}
  &\leq& \|\Delta_j f\|_{L^p}+C\|u\|_{\dot C^\delta} 2^{-j\delta} \sum_{k\leq j+4}2^k\| \Delta_k\theta\|_{L^p}+ \nonumber\\
  &&+\,C \|u\|_{\dot C^\delta} \,2^j \sum_{k\geq j-3}2^{-k\delta}\|\Delta_k\theta\|_{L^p} \nonumber\\
  & := & H_j^1 + H_j^2 + H_j^3.
\end{eqnarray}
Thus Gr\"onwall's inequality yields that for every $j\geq j_0$ and $t\geq 0$,
\begin{eqnarray}\label{theLp3}
  \|\Delta_j\theta(t)\|_{L^p} \leq e^{-\frac{c}{2}t 2^{j(\alpha-\sigma)}} \|\Delta_j\theta_0\|_{L^p}+ \int_0^t e^{-\frac{c}{2}(t-\tau)2^{j(\alpha-\sigma)}} \lp H_j^1(\tau) +H_j^2(\tau)+H_j^3(\tau)\rp\,\dd \tau.
\end{eqnarray}
According to Lemma \ref{lem:Lpes}, we also have the $L^p$-estimate for the considered equation \eqref{DD}:
\begin{eqnarray}\label{eq:Lpes}
  \|\theta(t)\|_{L^p}\leq e^{Ct} \lp\|\theta_0\|_{ L^p}+\int_0^t \|f(\tau)\|_{ L^p}\dd t\rp.
\end{eqnarray}
By arguing as \eqref{dataest}, we get that for all $t>0$, $j\in\N$ and $s\in(0,\alpha-\sigma)$,
\begin{eqnarray}\label{dataest2}
  2^{j s}e^{-\frac{c}{2}t2^{j(\alpha-\sigma)}} \|\Delta_j\theta_0\|_{L^p}  \leq  C_{\alpha,\sigma,s} t^{-\frac{s}{\alpha-\sigma}} \|\theta_0\|_{L^p},
\end{eqnarray}
thus collecting \eqref{theLp3}, \eqref{eq:Lpes} and \eqref{dataest2} leads to
\begin{eqnarray}\label{theBsp}
  \|\theta(t)\|_{ B^s_{p,\infty}}
  &\leq& \sup_{j\leq j_0} 2^{js}\|\Delta_j \theta(t)\|_{L^p} + \sup_{j\geq j_0}2^{js} \|\Delta_j\theta(t)\|_{L^p} \nonumber \\
  & \leq & C 2^{j_0} e^{Ct}  \lp \|\theta_0\|_{L^p}+ \|f\|_{L^1_t L^p} \rp + C_{\alpha,\sigma,s} t^{-\frac{s}{\alpha-\sigma}}\|\theta_0\|_{L^p} + \nonumber\\
  &&  + \sup_{j\geq j_0}\int_0^t e^{-\frac{c}{2}(t-\tau)2^{j(\alpha-\sigma)}} 2^{js} \lp H_j^1(\tau) + H_j^2(\tau) + H_j^3(\tau) \rp\dd \tau.
\end{eqnarray}
For the terms involving $H_j^1$, $H_j^2$ and $H_j^3$, in a similar way as obtaining \eqref{esFj1}, \eqref{esFj2} and \eqref{esFj3} respectively,
we have that for every $s\in (0,\alpha-\sigma)$ and $\delta\in(1-\alpha+\sigma,1)$,
\begin{eqnarray}\label{esHj1}
  \sup_{j\geq j_0}\int_0^t e^{-\frac{c}{2}(t-\tau)2^{j(\alpha-\sigma)}} 2^{js} H_j^1(\tau)\dd\tau
  \leq C \|f\|_{L^\infty_t\dot B^\delta_{p,\infty}} \sup_{j\geq j_0} 2^{j(s-\alpha +\sigma-\delta)} \leq C \|f\|_{L^\infty_t\dot B^\delta_{p,\infty}},
\end{eqnarray}
and
\begin{eqnarray}\label{esHj2}
  \sup_{j\geq j_0} \int_0^t e^{-\frac{c}{2}(t-\tau)2^{j(\alpha-\sigma)}} 2^{js} H_j^2(\tau)\dd\,\tau
  \leq C t^{-\frac{s}{\alpha-\sigma}}  2^{j_0(1-\alpha+\sigma-\delta)} \|u\|_{L^\infty_t \dot C^\delta}\lp \sup_{\tau\in (0,t]}\tau^{\frac{s}{\alpha-\sigma}}\|\theta(\tau)\|_{ B^s_{p,\infty}}\rp,
\end{eqnarray}
and
\begin{eqnarray}\label{esHj3}
  &&\sup_{j\geq j_0} \int_0^t e^{-\frac{c}{2}(t-\tau)2^{j(\alpha-\sigma)}} 2^{js} H_j^3(\tau)\,\dd\tau \nonumber \\
  &=& C \sup_{j\geq j_0} \int_0^t
  e^{-\frac{c}{2}(t-\tau)2^{j(\alpha-\sigma)}} \|u(\tau)\|_{\dot C^\delta} 2^{j(s+1)} \bigg(\sum_{k\geq j-3} 2^{-k\delta} \|\Delta_k \theta(\tau)\|_{L^p}\bigg) \dd\tau \nonumber\\
  &\leq &  C \|u\|_{L^\infty_t \dot C^\delta}\sup_{j\geq j_0} 2^{j(s+1)}\bigg(\sum_{k\geq j-3} 2^{-k(\delta+s)}\bigg)
  \int_0^t e^{-\frac{c}{2}(t-\tau)2^{j(\alpha-\sigma)}} \|\theta(\tau)\|_{B^s_{p,\infty}}\dd\tau \nonumber\\
  & \leq & C t^{-\frac{s}{\alpha-\sigma}}  2^{j_0(1-\alpha+\sigma-\delta)} \|u\|_{L^\infty_t \dot C^\delta}\lp \sup_{\tau\in (0,t]}\tau^{\frac{s}{\alpha-\sigma}}\|\theta(\tau)\|_{ B^s_{p,\infty}}\rp.
\end{eqnarray}
Plugging the estimates \eqref{esHj1}, \eqref{esHj2}, \eqref{esHj3} into \eqref{theBsp} yields that for any $0<s<\alpha-\sigma$ and $0<t\leq T$,
\begin{eqnarray}\label{theBsp2}
  t^{\frac{s}{\alpha-\sigma}} \|\theta(t)\|_{B_{p,\infty}^s} & \leq &
  C T^{\frac{s}{\alpha-\sigma}} e^{CT}  2^{j_0} \lp \|\theta_0\|_{L^p} + \|f\|_{L^1_T L^p}\rp + C_{\alpha,\sigma,s} \|\theta_0\|_{L^p}+
  C T^{\frac{s}{\alpha-\sigma}}\|f\|_{L^\infty_T \dot B^\delta_{p,\infty}} + \nonumber\\
  && +\, C 2^{j_0(1-\alpha+\sigma-\delta)} \|u\|_{L^\infty_T \dot C^\delta} \lp \sup_{t\in (0,T]} t^{\frac{s}{\alpha-\sigma}}\|\theta(t)\|_{ B^s_{p,\infty}}\rp.
\end{eqnarray}
Now, by choosing $j_0\in \N$ such that $C 2^{j_0(1-\alpha+\sigma-\delta)} \|u\|_{L^\infty_T \dot C^\delta}\leq \frac{1}{2}$ and \eqref{j0-p1} holds, or more precisely,
\begin{equation}\label{j0p}
  j_0 := \max\set{\Big[ \frac{1}{\delta-(1-\alpha+\sigma)}\log_2 \lp2C \|u\|_{L^\infty_T \dot C^\delta}\rp\Big], \Big[\frac{1}{\alpha-\sigma}\log_2\Big( \frac{\widetilde{C}}{c}\Big)\Big]}+1,
\end{equation}
we have that for all $0<s<\alpha-\sigma$,
\begin{eqnarray}
  \sup_{t\in (0,T]} \lp t^{\frac{s}{\alpha-\sigma}}\|\theta(t)\|_{ B^s_{p,\infty}} \rp \leq  C( T + 1)\lp  e^{CT} 2^{j_0}\big(\|\theta_0\|_{L^p} + \|f\|_{L^1_T L^p}\big) +
  \|f\|_{L^\infty_T\dot B^\delta_{p,\infty}} \rp,
\end{eqnarray}
which implies that for arbitrarily small $t_0\in (0,T)$ and every $s_0\in (0,\alpha-\sigma)$,
\begin{eqnarray}
  \sup_{t\in [t_0,T]} \|\theta(t)\|_{ B^{s_0}_{p,\infty}} \leq C t_0^{-\frac{s_0}{\alpha-\sigma}}( T + 1) \lp e^{CT} 2^{j_0}\big(\|\theta_0\|_{L^p} + \|f\|_{L^1_T L^p}\big) +
  \|f\|_{L^\infty_T\dot B^\delta_{p,\infty}} \rp,
\end{eqnarray}
where $j_0$ is given by \eqref{j0p}.

\textbf{Step 2:} the estimation of $\|\theta\|_{L^\infty([t_1,T]; B^{s_0+s_1}_{p,\infty})}$ for every $s_0,s_1\in (0,\alpha-\sigma)$ and $t_1\in (t_0,T)$.

For every $j\geq j_0$ with $j_0\in\N$ satisfying \eqref{j0-p1} chosen later, adapting the Gr\"onwall inequality to (\ref{thejLp2}) over the time interval $[t_0, t]$ (for $t>t_0>0$) yields
\begin{equation}\label{theLinf4}
  \|\Delta_j\theta(t)\|_{L^p} \leq e^{-\frac{c}{2} (t-t_0) 2^{j(\alpha-\sigma)}} \|\Delta_j\theta(t_0)\|_{L^p}
  + \int_{t_0}^t e^{-\frac{c}{2}(t-\tau)2^{j(\alpha-\sigma)}} \lp H_j^1(\tau) + H_j^2(\tau)+ H_j^3(\tau)\rp\,\dd \tau.
\end{equation}
By arguing as \eqref{dataest1}, we deduce that for $j\in \N$, $s_0\in(0,\alpha-\sigma)$ and every $s\in (0,\alpha-\sigma)$,
\begin{eqnarray}
  e^{-\frac{c}{2}(t-t_0)2^{j(\alpha-\sigma)}}2^{j (s_0 + s)} \|\Delta_{j}\theta(t_{0})\|_{L^p} \leq  C_{\alpha,\sigma,s} (t-t_0)^{-\frac{s}{\alpha-\sigma}}\|\theta(t_0)\|_{B_{p,\infty}^{s_0}},
\end{eqnarray}
thus we get that for all $t\geq t_0>0$,
\begin{eqnarray}\label{key-estp2}
  \|\theta(t)\|_{ B^{s_0+s}_{p,\infty}}
  &\leq& \sup_{j\leq j_0} 2^{j(s_0+s)}\|\Delta_j \theta(t)\|_{L^p} + \sup_{j\geq j_0}2^{j(s_0+s)} \|\Delta_j\theta(t)\|_{L^p} \nonumber \\
  & \leq & C 2^{j_0(s_0+s)} e^{Ct} \lp \|\theta_0\|_{L^p}+ \|f\|_{L^1_t L^p} \rp + C_{\alpha,\sigma,s} (t-t_0)^{-\frac{s}{\alpha-\sigma}}\|\theta(t_0)\|_{B^{s_0}_{p,\infty}} + \nonumber\\
  &&  + \sup_{j\geq j_0}\int_{t_0}^t e^{-\frac{c}{2}(t-\tau) 2^{j(\alpha-\sigma)}} 2^{j(s_0+s)} \lp H_j^1(\tau) + H_j^2(\tau) + H_j^3(\tau) \rp\dd \tau.
\end{eqnarray}
In a similar fashion of the estimating of \eqref{esFj1-2},\eqref{esFj2-2}-\eqref{esFj2-22} and \eqref{esFj3-2}, we find that for every $s\in(0,\alpha-\sigma)$ and $s_0+s<\delta+\alpha-\sigma$,
\begin{eqnarray}\label{esHj1-2}
  \sup_{j\geq j_0}\int_{t_0}^t e^{-\frac{c}{2}(t-\tau)2^{j(\alpha-\sigma)}} 2^{j(s_0+s)} H_j^1(\tau)\,\dd\tau \leq C \|f\|_{L^\infty_t \dot B^\delta_{p,\infty}},
\end{eqnarray}
and
\begin{eqnarray}\label{esHj2-2}
  && \sup_{j\geq j_0} \int_{t_0}^t e^{-\frac{c}{2}(t-\tau)2^{j(\alpha-\sigma)}} 2^{j(s_0+s)} H_j^2(\tau)\,\dd\tau \nonumber \\
  & \leq &
  \begin{cases}
  C \|u\|_{L^\infty_t \dot C^\delta}\lp \sup_{\tau\in (t_0,t]}(\tau-t_0)^{\frac{s}{\alpha-\sigma}} \|\theta\|_{ B^{s_0+s}_{p,\infty}}\rp
  (t-t_0)^{-\frac{s}{\alpha-\sigma}}  2^{j_0 \frac{1-\alpha+\sigma-\delta}{2}}, &\,\textrm{if   }0< s_0+s\leq 1\\
  C \|u\|_{L^\infty_t \dot C^\delta}\lp \sup_{\tau\in (t_0,t]}(\tau-t_0)^{\frac{s}{\alpha-\sigma}}\|\theta\|_{ B^{s_0+s}_{\infty,\infty}}\rp
  (t-t_0)^{-\frac{s}{\alpha-\sigma}}  2^{j_0(1-\alpha+\sigma-\delta)} , &\,\textrm{if  }1<s_0+s < \delta+\alpha-\sigma,
  \end{cases}
\end{eqnarray}
and for all $s\in (0,\alpha-\sigma)$,
\begin{eqnarray}\label{esHj3-2}
  &&\sup_{j\geq j_0} \int_{t_0}^t e^{-\frac{c}{2}(t-\tau)2^{j(\alpha-\sigma)}} 2^{j(s_0+s)} H_j^3(\tau)\,\dd\tau \nonumber \\
  &=& C \sup_{j\geq j_0} \int_{t_0}^t
  e^{-\frac{c}{2}(t-\tau)2^{j(\alpha-\sigma)}} \|u(\tau)\|_{\dot C^\delta} 2^{j(s_0+s+1)} \bigg(\sum_{k\geq j-3} 2^{-k\delta} \|\Delta_k \theta(\tau)\|_{L^p}\bigg) \dd\tau \nonumber\\
  &\leq &  C \|u\|_{L^\infty_t \dot C^\delta}\sup_{j\geq j_0} 2^{j(s_0+s+1)}\bigg(\sum_{k\geq j-3} 2^{-k(\delta+s_0+s)}\bigg)
  \int_{t_0}^t e^{-\frac{c}{2}(t-\tau)2^{j(\alpha-\sigma)}} \|\theta(\tau)\|_{B^{s_0+s}_{\infty,\infty}}\dd\tau \nonumber\\
  &\leq & C \|u\|_{L^\infty_t \dot C^\delta} \lp \sup_{\tau\in (t_0,t]}(\tau-t_0)^{\frac{s}{\alpha-\sigma}}\|\theta(\tau)\|_{ B^{s_0+s}_{p,\infty}}\rp
  \sup_{j\geq j_0} 2^{j(1-\delta)} \int_{t_0}^t e^{-\frac{c}{2}(t-\tau)2^{j(\alpha-\sigma)}}(\tau-t_0)^{-\frac{s}{\alpha-\sigma}}\dd\tau \nonumber \\
  & \leq & C \|u\|_{L^\infty_t \dot C^\delta}\lp \sup_{\tau\in (t_0,t]}(\tau-t_0)^{\frac{s}{\alpha-\sigma}}\|\theta(\tau)\|_{ B^{s_0+s}_{p,\infty}}\rp
  (t-t_0)^{-\frac{s}{\alpha-\sigma}}  2^{j_0(1-\alpha+\sigma-\delta)} .
\end{eqnarray}
Inserting the estimates \eqref{esHj1-2}-\eqref{esHj3-2} into \eqref{key-estp2}, we obtain that for every $t\in (t_0,T]$,
$s\in (0,\alpha-\sigma)$ and $ s_0+s < \delta+\alpha-\sigma$,
\begin{eqnarray*}
  && (t-t_0)^{\frac{s}{\alpha-\sigma}} \|\theta(t)\|_{B_{p,\infty}^{s_0+s}} \nonumber \\
  & \leq &
  C T^{\frac{s}{\alpha-\sigma}} e^{CT} \lp \|\theta_0\|_{L^p} + \|f\|_{L^1_T L^p}\rp 2^{j_0(s_0+s)} + C_{\alpha,\sigma,s} \|\theta(t_0)\|_{B^{s_0}_{p,\infty}}+
  C T^{\frac{s}{\alpha-\sigma}}\|f\|_{L^\infty_t\dot B^\delta_{p,\infty}} + \nonumber\\
  && + \,
  \begin{cases}
    C 2^{j_0\frac{1-\alpha+\sigma-\delta}{2}} \|u\|_{L^\infty_T \dot C^\delta} \lp \sup_{t\in (t_0,T]} (t-t_0)^{\frac{s}{\alpha-\sigma}}\|\theta\|_{B^{s_0+s}_{p,\infty}}\rp,
    \quad &\textrm{if}\;\; s_0+s \leq 1, \\
    C 2^{j_0(s_0+s-\alpha+\sigma-\delta)} \|u\|_{L^\infty_T \dot C^\delta}\lp \sup_{t\in (t_0,T]} (t-t_0)^{\frac{s}{\alpha-\sigma}}\|\theta\|_{B^{s_0+s}_{p,\infty}}\rp,
    \quad & \textrm{if}\;\; 1< s_0+s < \delta+\alpha-\sigma.
  \end{cases}
\end{eqnarray*}
Hence by selecting $j_0\in \N$ as
\begin{eqnarray}\label{j0-p2}
  j_0 :=
  \begin{cases}
    \;\max\set{\Big[\frac{2}{\delta-(1-\alpha+\sigma)}\log_2 \lp2C \|u\|_{L^\infty_T \dot C^\delta}\rp\Big], \Big[\frac{1}{\alpha-\sigma}\log_2\Big( \frac{2 \widetilde{C}}{c}\Big)\Big]}+1,
    \quad &\textrm{if}\;\; s_0+s \leq 1, \\
    \max\set{\Big[ \frac{1}{\delta +\alpha-\sigma-(s_0+s)}\log_2 \lp2C \|u\|_{L^\infty_T \dot C^\delta}\rp\Big], \Big[\frac{1}{\alpha-\sigma}\log_2\Big( \frac{2\widetilde{C}}{c}\Big)\Big]}+1,
    \quad & \textrm{if}\;\; 1< s_0+s < \delta+\alpha-\sigma,
  \end{cases}
\end{eqnarray}
we find that for all $s\in (0,\alpha-\sigma)$ and $ s_0+s < \delta+\alpha-\sigma$,
\begin{eqnarray*}
  && \sup_{t\in (t_0,T]} \lp (t-t_0)^{\frac{s}{\alpha-\sigma}}\|\theta(t)\|_{ B^{s_0+s}_{p,\infty}} \rp \nonumber\\
  & \leq &  C( T + 1) \lp \|\theta_0\|_{L^p}
  +\|f\|_{L^1_T L^p}\rp 2^{j_0(s_0+s)} + C\|\theta(t_0)\|_{B^{s_0}_{p,\infty}}  +  C(T+1)\|f\|_{L^\infty_T\dot B^\delta_{p,\infty}},
\end{eqnarray*}
which ensures that for any $t_1\in (t_0,T)$ and every $s_0, s_1 \in (0, \alpha-\sigma)$ satisfying $s_0+s_1<\delta+\alpha-\sigma$,
\begin{eqnarray}\label{esthe-ps0s1}
  \sup_{t\in [t_1,T]} \|\theta(t)\|_{ B^{s_0+s_1}_{p,\infty}} & \leq & C (t_1-t_0)^{-\frac{s_1}{\alpha-\sigma}}\lp ( T + 1) e^{CT} \big(\|\theta_0\|_{L^p} + \|f\|_{L^1_T L^p}\big)2^{j_0(s_0+s_1)} + \|\theta(t_0)\|_{B^{s_0}_{p,\infty}}\rp + \nonumber\\
  && + \, C(t_1-t_0)^{-\frac{s_1}{\alpha-\sigma}}(T+1)\|f\|_{L^\infty_T \dot B^\delta_{p,\infty}},
\end{eqnarray}
where $j_0$ is given by \eqref{j0-p2}.

\textbf{Step 3: }the estimation of $\|\theta\|_{L^\infty([t', T];C^{1,\gamma})}$ for some $\gamma>0$ and any $t'\in (0,T)$.

If $\alpha-\sigma \in (\frac{1}{2},1)$, we can choose appropriate indexes $s_0,s_1\in(0,\alpha-\sigma)$ so that $1<s_0+s_1<\delta+\alpha-\sigma$,
more precisely, denoting by $$\nu_1:=\min \set{\frac{2(\alpha-\sigma)-1}{2},\frac{\delta+\alpha-\sigma-1}{2}},$$
$s_0+s_1$ can be chosen so that $s_0+s_1=1+\nu_1$, thus in view of \eqref{esthe-ps0s1}, we obtain that
\begin{eqnarray}\label{eq:B1nu1}
  \sup_{t\in [t_1,T]}\|\theta(t)\|_{B^{1+\nu_1}_{p,\infty}}\leq C<\infty.
\end{eqnarray}
If $p>\frac{d}{\nu_1}$, then from the Besov embedding $B^{1+\nu_1}_{p,d}\hookrightarrow B^{1+\nu_1-\frac{d}{p}}_{\infty,\infty}$, we get the bound of $\|\theta\|_{L^\infty([t',T]; C^{1,\gamma})}$
with $t'=t_1$ and $\gamma=1+\nu_1-\frac{d}{p}>0$. If $p\leq \frac{d}{\nu_1}$, and we have the embedding $B^{1+\nu_1}_{p,\infty}\hookrightarrow L^{p_1}$ with some $p_1>\frac{d}{\nu_1}$,
by repeating the above \textbf{Step 1} and \textbf{Step 2} with $p_1$ in place of $p$, we can obtain the estimate of $\|\theta\|_{L^\infty([t_1^1, T]; B^{1+\nu_1}_{p_1,\infty})}$ with any $t_1^1\in (t_1,T)$,
which implies the bound of $\|\theta\|_{L^\infty([t_1^1,T];C^{1,\gamma})}$ with $\gamma = 1+\nu_1-\frac{d}{p_1}$.
Otherwise, for $p\leq \frac{d}{\nu_1}$ and $p_1$ satisfying $\frac{d}{p_1}=\frac{d}{p}-(1+\nu_1)$ is such that $p_1 \in(p,\frac{d}{\nu_1}]$,
as above we can obtain the bound of $\|\theta\|_{L^\infty([t_1^1, T]; B^{1+\nu_1}_{p_1,\infty})}$ with any $t_1^1\in (t_1,T)$,
then if the embedding $B^{1+\nu_1}_{p_1,\infty}\hookrightarrow L^{p_2}$ with some $p_2 >\frac{d}{\nu_1}$, we can repeat the above \textbf{Step 1} and \textbf{Step 2}
to conclude the proof, while if $p_2$ satisfying $\frac{d}{p_2}=\frac{d}{p_1}-(1+\nu_1)=\frac{d}{p}-2(1+\nu_1)$ is still such that $p_2\in (p_1, \frac{d}{\nu_1}]$,
we can iterate the above steps for several times, say $m$-times,
to find some number $p_{m+1}> \frac{d}{\nu_1}$ and obtain the bound of $\|\theta\|_{L^\infty([t^{m+1}_1,T ]; B^{1+\nu_1}_{p_{m+1},\infty})}$ with $t_1^{m+1}\in (t_1^m,T)$ any chosen,
which further implies the bound of $\|\theta\|_{L^\infty([t_1^{m+1},T];C^{1,\gamma})}$ with $\gamma = 1+\nu_1-\frac{d}{p_{m+1}}$.

For $\alpha-\sigma\in (0,\frac{1}{2}]$, we need to iterate the above procedure in \textbf{Step 2} for more times.
Assume that for some small number $t_k>0$, $k\in\N$, we already have a finite bound on $\|\theta(t_k)\|_{B^{s_0+s_1+\cdots+s_k}_{p,\infty}}$ with
$s_0,s_1,\cdots,s_k\in (0,\alpha-\sigma)$ satisfying $s_0+s_1+\cdots+s_k\leq 1$,
then by arguing as \eqref{esthe-ps0s1}, we deduce that for any $t_{k+1}>t_k$, $s_{k+1}\in (0,\alpha-\sigma)$ satisfying $s_0+s_1+\cdots+s_{k+1}<\delta+\alpha-\sigma$,
\begin{eqnarray}\label{esthe-ps0sk}
  && \sup_{t\in [t_{k+1},T]} \|\theta(t)\|_{ B^{s_0+s_1+\cdots+s_{k+1}}_{p,\infty}} \nonumber \\
  & \leq & C (t_{k+1}-t_k)^{-\frac{s_{k+1}}{\alpha-\sigma}}\lp ( T + 1) \big( \|\theta_0\|_{L^p} + \|f\|_{L^1_T L^p}\big) 2^{j_0(s_0+s_1+\cdots+s_{k+1})}  + \|\theta(t_k)\|_{B^{s_0+s_1+\cdots+s_k}_{p,\infty}}\rp + \nonumber\\
  && + \, C(t_{k+1}-t_k)^{-\frac{s_{k+1}}{\alpha-\sigma}}(T+1)\|f\|_{L^\infty_T \dot B^\delta_{p,\infty}},
\end{eqnarray}
where $j_0$ is also given by \eqref{j0-p2} with $s_0+s_1$ replaced by $s_0+s_1+\cdots+s_{k+1}$. Hence if $\alpha-\sigma\in (\frac{1}{k+2},\frac{1}{k+1}]$, $k\in\N^+$,
we can select appropriate numbers $s_0,s_1,\cdots,s_{k+1}\in (0,\alpha-\sigma)$ so that $1< s_0+s_1+ \cdots+ s_{k+1} <\delta +\alpha-\sigma$,
or, more precisely, $s_0+s_1+\cdots + s_{k+1}= 1+ \nu_{k+1}$, with
\begin{equation*}
  \nu_{k+1} := \min\set{\frac{(k+2)(\alpha-\sigma)-1}{2}, \frac{\delta+\alpha-\sigma-1}{2}},
\end{equation*}
and by repeating \textbf{Step 2} in the above manner for $(k+1)$-times, we obtain
\begin{eqnarray}\label{esthep1+}
  \sup_{t\in [t_{k+1},T]}\|\theta(t)\|_{B^{1+\nu_{k+1}}_{p,\infty}}\leq C <\infty.
\end{eqnarray}
The following deduction is similar to that stated below \eqref{eq:B1nu1}. If $p>\frac{d}{\nu_{k+1}}$, then from
$B^{1+\nu_{k+1}}_{p,\infty}\hookrightarrow B^{1+\nu_{k+1}-\frac{d}{p}}_{\infty,\infty}$, we naturally get the estimate of $\|\theta\|_{L^\infty([t_{k+1},T]; C^{1,\gamma})}$
with $\gamma= 1+\nu_{k+1}-\frac{d}{p}$. Otherwise, there exists a unique number $m\in \N$ so that
\begin{equation}\label{m-cd}
  \frac{d}{p}-m(1+\nu_{k+1})\geq \nu_{k+1},\quad\textrm{and}\quad\frac{d}{p}-(m+1)(1+\nu_{k+1})<\nu_{k+1},
\end{equation}
and by denoting $p_j\in[p,\infty)$ by
\begin{equation*}
  \frac{d}{p_j}=\frac{d}{p} - j (1+\nu_{k+1}),\quad j=0,1,2,\cdots, m,
\end{equation*}
we see that $p=p_0<p_1<\cdots<p_m\leq \frac{d}{\nu_{k+1}}$, thus by repeating the above process in obtaining \eqref{esthep1+} with $p_j$ replaced by $p_{j+1}$ iteratively ($j=0, 1,\cdots,m-1$),
we have the bound of $\|\theta\|_{L^\infty([t_{k+1}^m, T]; B^{1+\nu_{k+1}}_{p_m,\infty})}$ with any $t_{k+1}^m \in(0,T)$ (with the convention $t^0_j:=t_j$ for $j=0,1,\cdots,k+1$),
which ensures that there is some $p_{m+1}>\frac{d}{\nu_{k+1}}$ so that $\|\theta\|_{L^\infty([t_{k+1}^m, T]; L^{p_{m+1}})}$ is bounded,
and then iterating the above process once again leads to the estimate of $\|\theta\|_{L^\infty([t_{k+1}^{m+1}, T]; B^{1+\nu_{k+1}}_{p_{m+1},\infty})}$
with any $t_{k+1}^{m+1}\in(t_{k+1}^m,T)$ and moreover implies that for $\gamma =1+\nu_{k+1}-\frac{d}{p_{m+1}}$,
\begin{eqnarray}\label{esthepC1+}
  \|\theta\|_{L^\infty([t_{k+1}^{m+1},T]; C^{1,\gamma})}\approx \|\theta\|_{L^\infty([t_{k+1}^{m+1},T]; B^{1+\gamma}_{\infty,\infty})} \leq
  C \lp \|\theta_0\|_{L^p} + \|f\|_{L^\infty_T (B^\delta_{p,\infty}\cap B^\delta_{\infty,\infty})}\rp,
\end{eqnarray}
where $C>0$ is a constant depending only on $p,\alpha,\sigma,\delta,d,T$, $t_j^i$ ($i=0,1,\cdots,m+1$, $j=0,1,\cdots,k+1$) and $\|u\|_{L^\infty_T \dot C^\delta}$.

Therefore, for every $\alpha\in (0,1]$, $\sigma\in [0,\alpha)$, $p\in [2,\infty)$, and for any $t'\in (0,T)$,
there is some $k\in \N$ so that $\alpha-\sigma\in (\frac{1}{k+2},\frac{1}{k+1}]$,
and there is some number $m\in \N$ so that \eqref{m-cd} holds,
and thus we can choose $t_i^j= \frac{j(k+2)+i+1}{(k+2)(m+2)} t'$ for $i=0,1,\cdots,k+1$, $j=0,1,2,\cdots,m+1$,
and appropriate numbers $s_0,s_1,\cdots,s_{k+1}\in (1-\delta,\alpha-\sigma)$ such that
$s_0+s_1+\cdots+s_{k+1}=1+\nu_{k+1}$, then according to \eqref{esthepC1+} we iteratively prove the a priori estimate of \eqref{target2}.

\subsection{The existence issue}\label{subsec:thm2-3}

This part is similar to the deduction in the subsection \ref{subsec:ext1}. We also consider the approximate system \eqref{appDD},
and due to that $\|\theta_{0,\epsilon}\|_{H^s(\R^d)}\lesssim_\epsilon \|\theta_0\|_{L^p(\R^d)}$ for all $s\geq 0$,
we similarly obtain a smooth approximate solution $\theta_\epsilon\in C([0,T]; H^s(\R^d))\cap C^\infty((0,T]\times\R^d)$, $s>\frac{d}{2}+1$ for the system \eqref{appDD}.

Noting that we have the following uniform-in-$\epsilon$ estimates that $\|\theta_{0,\epsilon}\|_{L^p}\leq \|\theta_0\|_{L^p}$,
$\|u_\epsilon\|_{L^\infty_T \dot C^\delta}\leq \|u\|_{L^\infty_T \dot C^\delta}$ and
$\|f_\epsilon\|_{L^\infty_T (B^\delta_{p,\infty}\cap B^\delta_{\infty,\infty})}\leq \|f\|_{L^\infty_T (B^\delta_{p,\infty}\cap B^\delta_{\infty,\infty})}$,
we consider the equation of $\theta_\epsilon$ and by arguing as \eqref{esthepC1+} in the above, we can obtain the uniform-in-$\epsilon$ estimate of
$\|\theta_\epsilon\|_{L^\infty((0,T]; C^{1,\gamma}(\mathbb{R}^d))}$ with some $\gamma>0$.
Such a uniform estimate ensures that up to a subsequence, $\theta_\epsilon$ pointwisely converges to a function $\theta$ on $(0,T]\times \R^d$,
and we also have $\theta\in L^\infty((0,T];C^{1,\gamma}(\R^d))$ which satisfies \eqref{target2}. By passing $\epsilon$ to $0$ in \eqref{appDD}, we can deduce that $\theta$ is a distributional solution of \eqref{DD}.

%
\vskip 0.2in
\textbf{Acknowledgements.}
L. Xue was supported by NSFC grant No. 11401027 and Youth Scholars Program of Beijing Normal University.

\end{document}